\documentclass[12pt,reqno]{amsproc}

\usepackage{mathpazo, mathrsfs}

\usepackage{amsmath,amssymb,amsthm}
\usepackage{enumitem}
\usepackage{array}
\usepackage{booktabs}
\usepackage{tikz}
\usetikzlibrary{arrows.meta,decorations.markings,positioning,calc}
\usepackage[hidelinks,colorlinks=true,linkcolor=blue,citecolor=blue]{hyperref}

\hypersetup{pdftitle={Landfall Geometry},pdfauthor={Frances Kirwan and Vidit Nanda}}
\usepackage[capitalise,noabbrev]{cleveref}
\usepackage[margin=.8in]{geometry}
\renewcommand{\mathcal}{\mathscr}

\newcommand{\Crit}{\operatorname{\bf Crit}}
\newcommand{\Inv}{\operatorname{\bf Inv}}

\newcommand{\dist}{\operatorname{\bf dist}}
\newcommand{\Hess}{\operatorname{Hess}}
\newcommand{\ev}{\operatorname{\bf ev}}
\newcommand{\dest}{\operatorname{dest}}
\newcommand{\Ob}{\operatorname{Ob}}
\newcommand{\Mor}{\operatorname{Mor}}

\newcommand{\Con}{\operatorname{\bf Con}}
\newcommand{\set}[1]{\left\{{#1}\right\}}
\newcommand{\ip}[1]{\left\langle{#1}\right\rangle}

\newcommand{\Z}{\mathbb Z}
\newcommand{\R}{\mathbb R}
\newcommand{\C}{\mathbb C}
\newcommand{\eps}{\epsilon}
\newcommand{\Disk}{\mathbb D}
\newcommand{\cK}{\mathcal K}

\newcommand{\cL}{\mathcal L}

\newcommand{\Ev}{\operatorname{Ev}}            % evaluation map B\cL -> X
\newcommand{\Kord}{\mathbf K}                  % order complex \Kord(P)
\newcommand{\kk}{\Bbbk}                        % ground field (\mathbb{k} has no glyph under mathpazo)
\newcommand{\Sing}{\mathsf S}                  % singular chains \Sing_*(X)
\newcommand{\RT}{\mathbf T}                    % Reidemeister torsion
\newcommand{\simp}{\mathbf\Delta}              % standard simplex \simp^k

\theoremstyle{plain}
\newtheorem{theorem}{Theorem}[section]
\newtheorem{proposition}[theorem]{Proposition}
\newtheorem{corollary}[theorem]{Corollary}
\newtheorem{lemma}[theorem]{Lemma}
\newtheorem*{theorem*}{Theorem}
\theoremstyle{definition}
\newtheorem{definition}[theorem]{Definition}
\newtheorem{remark}[theorem]{Remark}
\newtheorem{example}[theorem]{Example}

\title{Landfall Geometry}
\subjclass[2020]{Primary 55U15; Secondary 37B30, 55P05, 03C64, 58K05}
\keywords{Conley index, Morse theory, flow category, connection matrix, gradient flow, discrete Morse theory}
\author{Frances Kirwan}
\author{Vidit Nanda}
\date{}

\begin{document}
	
	\begin{abstract}
		We study the negative gradient flow of analytic maps $f:X\to\R$ defined on compact analytic Riemannian manifolds without imposing any nondegeneracy or transversality assumptions. We construct tight isolating blocks around the connected components of the critical set $\Crit(f)$ and record only bounded-time contacts made after an orbit exits one block and before it makes landfall within the interior of another. These contacts assemble into a strict topological category $\cL$ -- called the landfall category -- whose classifying space $B\cL$ is shown to be simple homotopy equivalent to $X$ naturally over all down-sets of the flow partial order. The filtered homotopy type is independent of the isolating blocks. Passing to a finite subanalytic triangulation yields a flow-induced acyclic partial matching, and hence, an explicit finite chain complex generated by the non-exit simplices of the blocks which computes the homology of $X$. The finite model also computes the Morse--Conley spectral sequence over arbitrary coefficients; over a field, homological perturbation further produces a connection matrix whose entries are finite sums over descending contact itineraries.
	\end{abstract}
	
	\maketitle
	
	\section*{Introduction}
	
	Let $X$ be a compact Riemannian manifold and $f:X \to \R$ a smooth map. For any pair of real numbers $c<d$, we have a tautological {\em adjunction pair} of spaces $(D,S)$ which relates the sublevel sets of $f$ at $c$ and $d$. In particular, the equality
	\[
	\set{f \leq d} = \set{f \leq c} \cup_S D
	\]
	evidently holds for $D = \set{c \leq f \leq d}$ and $S = \set{f=c}$ regardless of $f$. If $f$ is Morse with a unique critical point $x$ in $\set{c\leq f \leq d}$, and $c<f(x)<d$, then $\set{f\leq d}$ is homotopy equivalent to $\set{f\leq c}$ with a single cell attached, whose dimension is the Morse index of $x$ \cite[Theorem~3.2]{Milnor1963}. If $f$ further satisfies Morse--Smale transversality, then there is a chain complex computing the homology of $X$ whose chain groups are generated by critical points and whose boundary operator is built from counting gradient flow lines \cite{Witten1982}.  The flow category formalism from \cite{CohenJonesSegalClassifying} assembles compactified moduli spaces of gradient trajectories into a category whose classifying space can reconstruct the homotopy type of $X$.
	
	None of these miracles survives unscathed once we relax the underlying nondegeneracy and transversality assumptions on $f$ --- critical sets need not be submanifolds (let alone isolated points), and the homotopy types of the adjunction pairs across critical values are encoded by index pairs \cite{Conley1978}, whose quotient homotopy type is in general too complicated to be represented by a single integer such as the Morse index. Even when such an adjunction pair is understood, its attaching map to the lower sublevel set need not decompose into well-behaved pieces indexed by the lower critical components. The moduli spaces of compactified flow lines also no longer have the regularity properties necessary to support Morse boundary operators or topologically faithful flow categories. Moreover, the usual reconstruction statement requires tameness hypotheses on these moduli spaces: the authors of \cite{CalleLiu2026} establish reconstruction under a local-contractibility hypothesis and exhibit a counterexample to the unrestricted claim. In this paper, we show that all the attaching data can be resolved for real-analytic $f:X \to \R$ without ever seeking recourse to asymptotic trajectory spaces.  
	
	\subsection*{The Landfall Category} We assume that the manifold $X$, together with its Riemannian metric $g$, and the function $f:X \to \R$ are analytic\footnote{\emph{Analytic} and \emph{real-analytic} are used interchangeably throughout.}\!. We allow $\Crit(f)$ to be singular and study the negative gradient flow of $f$ without imposing  Morse--Bott or Morse--Smale conditions. Let $\coprod_{p \in P} S_p$ be the decomposition of $\Crit(f)$ into connected components, and write $p \leadsto q$  whenever  $p\ne q$ and there exists some $x \in X$ whose forward and backward limits under the flow live in $S_q$ and $S_p$ respectively. Denote by $>_F$ the flow partial order on $P$ which arises from taking the transitive closure of $\leadsto$. We may, and will, choose disjoint, compact, subanalytic {\bf isolating blocks} $N_p$ so that each exit set $E_p \subset \partial N_p$ is the intersection of $N_p$'s boundary with a regular level set of $f$. 
	
	We follow the orbit of each $x\in E_p$ only until it first enters the {interior} of another block $N_r$. Before this terminal \emph{landfall}, it may touch the boundaries of several other blocks $N_q$. For each such $q$, let $\mathcal C_{qp}\subset E_p$ be the (necessarily compact) set of all $x$ whose forward time flowlines intersect $\partial N_q$ before or at terminal landfall. The flow yields a continuous embedding $\chi_{qp}:\mathcal C_{qp} \to \partial N_q$; and for sufficiently small isolating blocks, $\mathcal C_{qp} \neq \varnothing$ implies $p >_F q$. In particular, every block $N_q$ touched on the way to landfall in $N_r$ satisfies $p >_F q \geq_F r$.
	
	The basic geometric fact underlying our construction here is that this cover $E_p = \bigcup_{p >_F q} \mathcal C_{qp}$ of each exit set is well-behaved along overlaps in the following sense. Whenever some $x$ lies in $\mathcal{C}_{qp} \cap \mathcal{C}_{rp}$ for  $r\ne q$, then (up to permuting $q$ and $r$) we must have $p >_F q >_F r$ and the image $y := \chi_{qp}(x)$ must lie in $\mathcal C_{rq}$. Moreover, $y$'s forward orbit (and therefore, its landfall destination) must coincide with the tail of $x$'s forward orbit. The resulting composition is illustrated by the striped arrow below:
	
	\begin{center}
		\includegraphics[scale=.45]{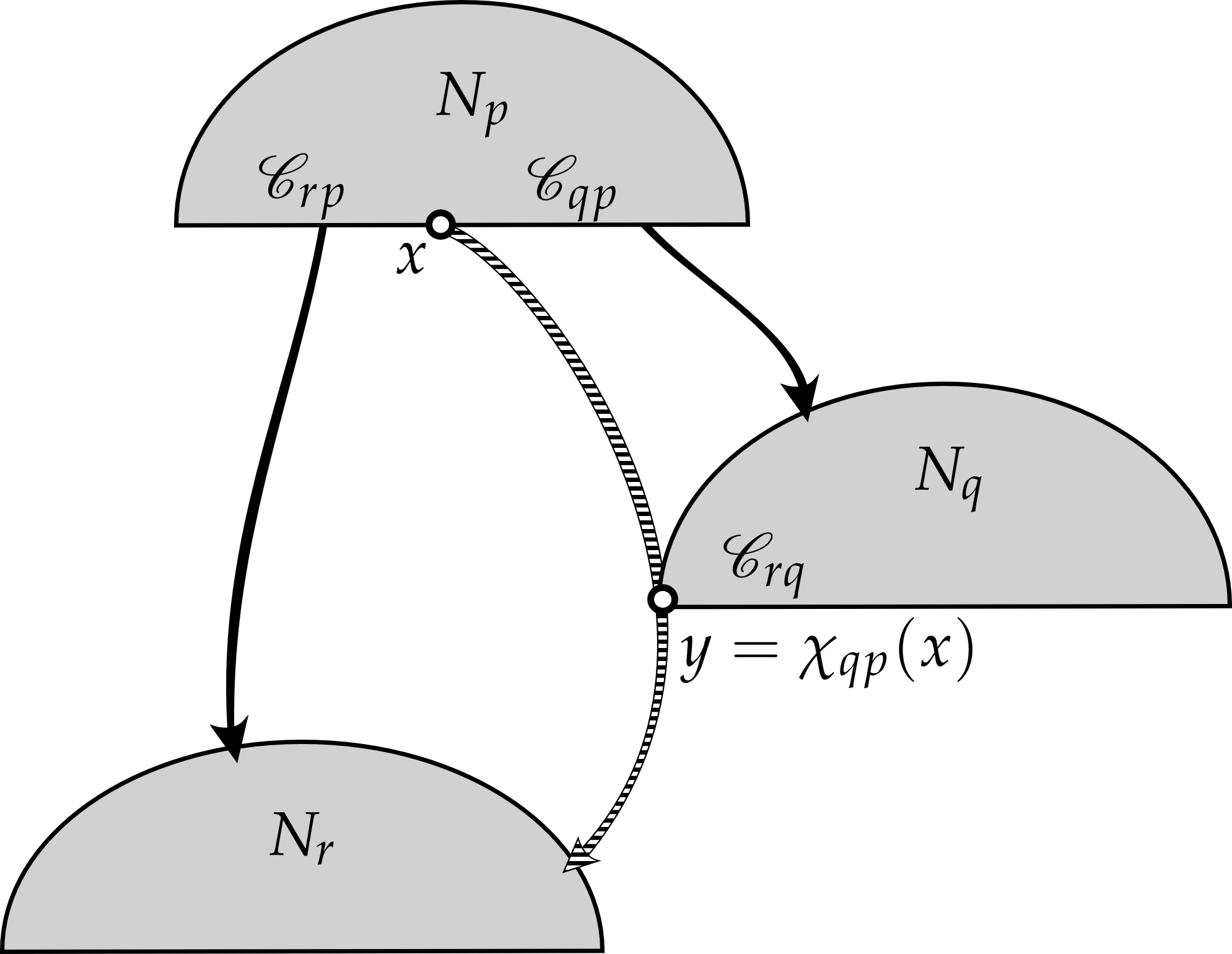}
	\end{center}
	
	\noindent (See \Cref{thm:factorisation,thm:contact-support}.) Armed with this picture, we define the {\bf landfall category} $\cL$ of $f$ with respect to the chosen isolating blocks $\set{N_p \mid p \in P}$ via:
	\[
	\Ob(\cL)=\coprod_{p\in P}N_p,
	\quad \text{ and } \quad
	\Mor(\cL)=\coprod_{p>_F q}\mathcal C_{qp},
	\]
	with source map $x\mapsto x$ and target map $x\mapsto\chi_{qp}(x)$. This is a nonunital category internal to topological spaces; composition is well-defined and strictly associative precisely because every composable pair of morphisms only implicates points lying along a single unbroken orbit of the gradient flow.
	
	\subsection*{Main Results} 
	
	Let $f:X \to \R$ be a real-analytic map and write $\coprod_{p \in P} S_p$ for the decomposition of its critical set into connected components. We show (in \Cref{cor:support-sandwich} below) that for sufficiently small isolating blocks $\set{N_p \mid p \in P}$, we have
	\[
	p \leadsto q \implies \mathcal{C}_{qp} \neq \varnothing \implies p >_F q,
	\]
	so the nonemptiness of hom-sets in the landfall category generates the flow-induced partial order on $\Crit(f)$ via transitive closure. Assuming that our $N_p$ are chosen sufficiently small in this sense, let $B\cL$ denote the geometric realisation of the semi-simplicial nerve \cite{EbertRandalWilliams2019} of the landfall category. Our first main result relates $B\cL$ to $X$.
	
	\begin{theorem*}[A]
		There is a simple homotopy equivalence $B\cL\simeq X$.
	\end{theorem*}
	
	Here $B\cL$ carries a finite regular CW structure induced by triangulations of the blocks (see \Cref{lem:prismatic-cw}). We show in \Cref{thm:classifying-space} below that this equivalence arises from an explicit {\bf evaluation map} $\Ev:B\cL \to X$, which respects the natural filtration of $X$ by the down-sets of $P$. Explicitly, for each down-set $A \subset P$, we have the subset $X_A \subset X$ consisting of all $x$ which lie on a forward orbit emanating from $N_p$ for some $p \in A$, as well as the full subcategory $\cL_A$ of $\cL$ spanned by objects $\coprod_{p \in A} N_p$. Now $\Ev$ restricts naturally to a simple homotopy equivalence $B\cL_A \simeq X_A$. It follows that for any pair $A' \subset A$ of down-sets in $P$, there is a pointed homotopy equivalence between $B\cL_{A}/B\cL_{A'}$ and the homotopy Conley index of the maximal flow-invariant subset of $\overline{X_{A} - X_{A'}}$. See \Cref{thm:filtered-classifying} and \Cref{thm:homotopy-conley} for details. For two choices of blocks, flowing for a sufficiently large uniform time gives a homotopy equivalence of their entire down-set filtrations. These comparison maps are canonical up to filtered homotopy and compatible with composition; see \Cref{prop:block-independence}. 
	
	Since $f$ is analytic and the $N_p$ are subanalytic, we may choose finite subanalytic triangulations of each $N_p$ compatible with the exit sets and contact maps. Let $Q_{p,*}=C_*(N_p,E_p;\Z)$ and write $d_p$ for its relative simplicial boundary map. These chain groups are freely generated by the simplices of $(N_p - E_p)$. For such a simplex $\alpha$, its ordinary boundary splits as $\partial\alpha=d_p\alpha+\beta_p\alpha$, where $\beta_p\alpha$ consists of the faces contained in $E_p$. Flowing each of these exit faces to its landfall block gives a linear map $L:\bigoplus_p C_*(E_p;\Z)\to\bigoplus_p Q_{p,*}$ which strictly lowers the flow order. Consider the direct sum $Q_*:=\bigoplus_p Q_{p,*}$, similarly set $d:=\bigoplus_p d_p$ and $\beta:=\bigoplus_p\beta_p$. Our next main result relates a deformation of $Q_*$ to the singular chain complex $\Sing_*(X;\Z)$ of $X$.
	
	\begin{theorem*}[B]
		The deformation $d^{\mathrm{LF}}=d+L\beta$ yields a boundary operator on $Q_*$, and the deformed chain complex $(Q_*,d^\text{\rm LF})$ is chain homotopy equivalent to $\Sing_*(X;\Z)$ naturally over the down-sets of $P$.
	\end{theorem*}
	
	\noindent (See \Cref{thm:landfall-complex,cor:filtered-chains} below.) The proof constructs an acyclic partial matching on a regular CW model of $B\cL$ and employs discrete Morse theory \cite{Forman1998, dmpers}. We call $(Q_*,d^\mathrm{LF})$ the {\bf landfall chain complex} associated to the chosen blocks and triangulations. Its boundary operator has a viscerally geometric description: take the ordinary relative boundary of a simplex, and send each discarded exit face to its landfall block. When $X$ is connected, the acyclic partial matching lifts  $\pi_1X$-equivariantly to a (finitely generated, based, free) $\Z[\pi_1X]$-complex which is simple chain homotopy equivalent to the cellular chains of the universal cover.  Thus the landfall complex retains torsion (see \Cref{ex:lens-torsion}). It also populates the {\em Morse--Conley spectral sequence}
	associated to any strictly order-preserving function $\ell:P\to \{0,\ldots,m\}$; filtering the landfall complex by $\ell$ produces
	\[
	E^1_{i,j}\cong\bigoplus_{\ell(p)=i}\Con_{i+j}(S_p;\Lambda)
	\quad\Longrightarrow\quad H_{i+j}(X;\Lambda)
	\]
	over any commutative ring $\Lambda$. For fixed $\ell$, this spectral sequence is canonically independent of the blocks and triangulations from $E^1$ onward, see \Cref{thm:spectral-sequence}.
	
	Our third main result gives a further compression when working over coefficients in a field $\kk$. We now take $Q_{p,*}=C_*(N_p,E_p;\kk)$, whose homology $V_{p,*}=H_*(N_p,E_p;\kk)$ is the homological Conley index of $S_p$, as described in \cite[\S III.5]{Conley1978}. Over a field, each $Q_p$ admits chain maps $\jmath_p:V_p\to Q_p$ and $\pi_p:Q_p\to V_p$, together with a degree $+1$ map $h_p:Q_p \to Q_p$ satisfying
	\[
	\pi_p\jmath_p=1,\quad \text{ and } \quad d_ph_p+h_pd_p=\jmath_p\pi_p-1,
	\]
	along with $h_p^2=h_p\jmath_p=\pi_ph_p=0$. 
	These data form a {\em normalised contraction} of chain complexes from each $(Q_p,d)$ onto its homology $(V_p,0)$, as in \cite{Skoldberg2006}. We write $Q = \bigoplus_p Q_p$ and $d = \bigoplus_p d_p$, etc.\ for direct sums over $P$.
	
	\begin{theorem*}[C]
		Choose a normalised contraction $(\jmath_p,\pi_p,h_p)$ of each $Q_p$ onto $V_p$, and write $\jmath,\pi,h$ for the corresponding direct sums over $P$. With $\Delta=L\beta$, the operator
		\[
		D=\pi\Delta(1-h\Delta)^{-1}\jmath
		\]
		is a connection matrix for the decomposition $\Crit(f)=\coprod_p S_p$. In particular, the homology of $(\bigoplus_p V_p,D)$ is $H_*(X;\kk)$.
	\end{theorem*}
	
	The inverse in this formula is a finite geometric series --- since $h$ preserves block labels while $\Delta$ strictly lowers the finite flow order, their composite $h\Delta$ is nilpotent. Thus every entry of $D$ is a finite sum over descending contact itineraries. The contraction can also be chosen over $\Z$ whenever every local homological Conley group is free abelian. The homological perturbation argument is standard \cite{Skoldberg2006} and appears in essentially the relevant filtered form in \cite{HarkerMischaikowSpendlove2021,TorrasCasasYimPennig2025}; our input is the bounded-time geometric construction of the filtered complex to which it applies. The connection-matrix assertion, in the sense of Franzosa \cite[\S 3]{Franzosa1989}, is proved in \Cref{thm:connection-matrix}.

	\subsection*{Context and References}
	
	While making no attempt to provide a comprehensive survey here, let us mention other attempts to extend Morse theory to various degenerate settings. Morse--Bott theory retains a controlled local model near a critical submanifold, and Morse--Bott complexes may be constructed
	using cascades \cite{BanyagaHurtubise2013}, while  \cite{KirwanPenington2021} treats smooth functions with finitely many connected critical components, without a nondegeneracy assumption. Corollary~4.6 of \cite{KirwanPenington2021} gives a spectral sequence with local relative homology on its first page. One may always deform the given degenerate function to a Morse--Smale one, but this approach becomes unsatisfactory when the original dynamics has structure which does not survive even small generic perturbations. A prime example of such structure is equivariance, which becomes indispensable in applications to Yang--Mills \cite{AtiyahBott1983} and geometric invariant theory \cite{Kirwan1984}. Our construction therefore works directly with the original analytic gradient flow and imposes no local normal form on its critical components.
	
	Many papers now build upon the topology of compactified moduli spaces of flow lines outlined in the original flow category preprint of Cohen, Jones and Segal \cite{CohenJonesSegalClassifying}. The impetus for this progress has largely come from profound applications in Floer theory \cite{CohenJonesSegal1995,Cohen2010,AbouzaidBlumberg2024}, but there are also purely combinatorial variants \cite{NandaTamakiTanaka2018, Nanda2019, Hem2024}. On the algebraic side, connection matrices and homology index braids originate in Conley--Franzosa theory \cite{Conley1978,Franzosa1989}, with more modern filtered-chain formulations in \cite{HarkerMischaikowSpendlove2021,TorrasCasasYimPennig2025}. Nicolaescu develops related reconstruction methods under tameness assumptions on the entire flow \cite{Nicolaescu2010}; here subanalyticity is required only on bounded time intervals. Both topology and algebra arise from the same bounded-time geometry which is captured by the landfall category at the level of spaces.
	
	\subsection*{Organisation}  \Cref{sec:blocks} constructs the block systems $\set{N_p}$ and develops the finite-time contact factorisation. \Cref{sec:category} constructs the landfall category and proves the classifying-space theorem. \Cref{sec:compression} performs the cellular Morse compression and identifies
	the landfall complex; the remainder of that section lifts this compression to the
	universal cover and extracts torsion.
	\Cref{sec:linearisation} identifies the interval Conley indices, proves block independence, constructs the spectral sequence and connection matrices, and specialises to the Morse--Bott and Morse settings. Finally, \Cref{sec:examples}
	gives examples.
	
	{\footnotesize
		\subsection*{Acknowledgements} The authors thank the Sydney Mathematics Research Institute for its hospitality. VN is partially supported by EPSRC grant EP/Z531224/1 and AFOSR contract FA9550-22-1-0462. This work was developed with assistance from OpenAI's ChatGPT 6 and Anthropic's Claude Opus 5.5.
	}
	
	\section{Isolating Blocks and Landfall}\label{sec:blocks}
	
	Let $(X,g)$ be a compact real-analytic Riemannian manifold and $f:X\to\R$ a real-analytic function. In this section we place disjoint isolating blocks around the connected components of $\Crit(f)$, follow each orbit from the exit set of one block until it makes landfall in the interior of another, and record which blocks it touches along the way. The resulting contact domains satisfy a truncation law, and for sufficiently tight blocks they recover the flow order on the critical components.
	
	\subsection{Analytic Gradient Flow}  Since $X$ is compact, the negative gradient vector field $-\nabla f$ is complete; we write $(t,x) \mapsto \phi_t(x)$ to indicate the negative gradient flow $\R\times X \to X$ of $f$. By uniqueness of solutions, the flow satisfies
	\[
	\phi_0(x) = x \quad \text{and} \quad \phi_s \circ \phi_t(x) = \phi_{s+t}(x)
	\]
	for all $s,t\in\R$ and $x\in X$. The {\bf critical set} $\Crit(f) \subset X$ consists of all points where $\nabla f$ vanishes.
	
	\begin{proposition}\label{prop:analytic-finiteness}
		With $(X,g)$ and $f$ as above, the following hold.
		\begin{enumerate}[label=\textup{(F\arabic*)},leftmargin=3.6em]
			\item The flow map $\phi:\R\times X\to X$ is real-analytic.
			\item For every $T \in [0,\infty)$, the graph of
			$\phi|_{[-T,T]\times X}$ is compact subanalytic.
			\item Along every nonconstant orbit, we have
			\begin{equation}\label{eq:descent}
				\frac{d}{dt}f \circ \phi_t(x) < 0.
			\end{equation}
			\item The set $\Crit(f)$ is analytic and has finitely many connected components, with $f$ being constant on each one.
		\end{enumerate}
	\end{proposition}
	
	\begin{proof}
		Analytic dependence of solutions on time and initial conditions \cite{Ilyashenko2008} gives the first assertion.  The second follows because an analytic map on the compact set $[-T,T]\times X$ has compact subanalytic graph \cite{BierstoneMilman1988}.  The third follows from the chain rule:
		\[
		\frac{d}{dt}f \circ \phi_t(x) = -\|\nabla f(\phi_t(x))\|^2 
		\]
		along with the fact that the right side equals zero for $x \in \Crit(f)$ and is strictly negative otherwise. The last assertion is standard for compact real-analytic sets; see
		\cite{KrantzParks2002}.
	\end{proof}
	
	We write the decomposition of $\Crit(f)$ into connected components as
	\begin{align}\label{eq:morse-decomposition}
		\Crit(f) = \coprod_{p \in P} S_p,
	\end{align} 
	noting that the indexing set $P$ is finite by assertion (F4) of \Cref{prop:analytic-finiteness}. By the same assertion, $f$ must attain a constant value $c_p$ on each $S_p$. We explicitly allow $c_p = c_q$ for $p \neq q$, and write $q \prec_c p$ whenever $c_q<c_p$. Since $f$ is analytic, the forward and backward limit sets of every point along the flow are singletons (rather than arbitrary connected invariant sets).
	
	\begin{proposition}\label{prop:capture}
		For every $x\in X$ there are unique critical points $\ell_\pm(x)$ such that
		\[
		\lim_{t\to\pm\infty}\phi_t(x)=\ell_\pm(x).
		\]
		If $\ell_-(x)\in S_p$ and $\ell_+(x)\in S_q$, then $c_p\ge c_q$, with equality occurring only for a constant orbit.
	\end{proposition}
	
	\begin{proof}
		The limits of $f \circ \phi_t(x)$ exist by compactness and monotonicity, and every accumulation point is critical.  The analytic gradient convergence theorem
		gives convergence to single points \cite{Lojasiewicz1984,KurdykaMostowskiParusinski2000}.  The order statement follows from \Cref{prop:analytic-finiteness}(F3).
	\end{proof}
	
	The following order on $P$ can relate far fewer pairs than the critical-value order $ \prec_c$, and it plays an important role in all of our constructions below.
	\begin{definition} \label{def:flow-order}
		For distinct $p,q\in P$, write $p\leadsto q$ if there exists some $x \in X$ with $\ell_-(x)$ in $S_p$ and $\ell_+(x)$ in $S_q$. The transitive closure $>_F$ of $\leadsto$ is called the {\bf flow-ordering} on $P$.
	\end{definition}
	Thus, $p >_F q$ holds whenever there is a finite chain in $P$ of the form
	\[
	p=p_0\leadsto p_1\leadsto\cdots\leadsto p_m=q.
	\]
	Since every $\leadsto$ relation strictly decreases the critical value, $>_F$ is a strict partial order; its reflexive closure is a partial order, and $p>_F q$ implies $q\prec_c p$. For each $N \subset X$ we write $\Inv(N)$ to indicate the maximal $\phi$-invariant subset of $N$; and given $\delta>0$, we define the open  neighbourhood
	\[
	B_\delta(N) := \set{x \in X \mid \dist_g(x,y) < \delta \text{ for some }y \in N},
	\]
	where $\dist_g$ denotes the distance on $X$ induced by its Riemannian metric $g$.
	
	\begin{lemma}\label{lem:critical-neighbourhoods}
		For all sufficiently small $\delta>0$, the sets
		$\set{\overline{B_\delta(S_p)} \mid p \in P}$ are pairwise disjoint and satisfy $\Inv(\overline{B_\delta(S_p)})=S_p$.
	\end{lemma}
	
	\begin{proof}
		The critical components are finitely many disjoint compact sets, so the first assertion holds automatically for sufficiently small $\delta$.  If a complete orbit is contained in $\overline{B_\delta(S_p)}$, \Cref{prop:capture} places both of 	its asymptotic limits in $S_p$.  They therefore have the same critical value, and strict descent \eqref{eq:descent} makes the orbit constant.
	\end{proof}
	
	It follows from the preceding result that \eqref{eq:morse-decomposition} is a {\bf Morse decomposition} for $\phi$ with respect to $>_F$, see \cite[\S I.7.2]{Conley1978}. The following result provides a convenient sufficient condition for distinct elements of $P$ to be related by this order.

	\begin{lemma}\label{lem:broken-flow-order}
		Given $p\ne q$ in $P$, suppose that there are sequences $\set{x_n} \subset X$ and $\set{T_n} \subset \R_{> 0}$ such that
		\[
		\dist_g(x_n,S_p) \to 0
		\quad \text{and} \quad
		\dist_g(\phi_{T_n}(x_n),S_q) \to 0.
		\]
		both hold as $n \to \infty$. Then $q<_F p$.
	\end{lemma}
	
	\begin{proof}
		Since $p\ne q$ and we have $f(x_n)\to c_p$ and $f(\phi_{T_n}(x_n))\to c_q$ as $n \to \infty$, the strict monotonicity of $f$ along every nonconstant orbit gives $c_q < c_p$. We fix $\delta>0$ as in 		\Cref{lem:critical-neighbourhoods} and proceed by induction on the cardinality $\#\set{r\in P\mid c_q\le c_r\le c_p}$. For large $n$, the point $x_n$ lies in $B_\delta(S_p)$ while
		$\phi_{T_n}(x_n)$ lies in the disjoint set $B_\delta(S_q)$.  Define
		\[
		\sigma_n := \sup\set{t \geq 0 \mid \phi_{[0,t]}(x_n)\subset B_\delta(S_p)}.
		\]
		Then $\sigma_n<T_n$.  If $\sigma_n$ were bounded along a subsequence, we could pass to limits $x_n\to x\in S_p$ and $\sigma_n\to\sigma$ and obtain $\phi_\sigma(x)\in\partial B_\delta(S_p)$, contradicting the invariance of $S_p$. Hence $\sigma_n\to\infty$.
		
		Set $z_n=\phi_{\sigma_n}(x_n)$ and pass to a subsequence converging to some $z$. Then $z\in\partial B_\delta(S_p)$, so $z$ is noncritical.  For every $s>0$ and all sufficiently large $n$ (depending on $s$), we have $\sigma_n\geq s$, and moreover $\phi_{-s}(z_n)$ lies in the closure of $B_\delta(S_p)$. Passing to the limit shows that the entire backward orbit of $z$ lies in this
		closure. By \Cref{prop:capture,lem:critical-neighbourhoods},
		$\ell_-(z)\in S_p$. Again by \Cref{prop:capture}, there is $r_1\in P$ with $\ell_+(z)\in S_{r_1}$.  Since $z$ is noncritical, we get
		$p\leadsto r_1$, so $r_1 <_F p$ and $c_{r_1}<c_p$.
		
		If $r_1=q$, then we are done, so we may assume henceforth that $r_1 \neq q$.
		We claim that $(T_n-\sigma_n) \to \infty$. Otherwise, after passing to a subsequence, we would have $(T_n-\sigma_n) \to T < \infty$, and $
		\phi_{T_n}(x_n)=\phi_{T_n-\sigma_n}(z_n) \to \phi_T(z)\in S_q$. 	Uniqueness for the flow would then make $z$ critical, a contradiction.
		
		Now  we can  choose times $s_k\to\infty$ with
		$\dist_g(\phi_{s_k}(z),S_{r_1})<\frac1{2k}$ and $n_k\ge k$ sufficiently large so that
		\[
		\dist_g\left(\phi_{s_k}(z_{n_k}),S_{r_1}\right)< k^{-1}
		\quad \text{and} \quad
		(T_{n_k}-\sigma_{n_k})>s_k.
		\]
		Set $x'_k=\phi_{\sigma_{n_k}+s_k}(x_{n_k})$ and $
		T'_k=T_{n_k}-\sigma_{n_k}-s_k>0$. Then $\dist(x'_k,S_{r_1})\to0$ and
		\[
		\phi_{T'_k}(x'_k)=\phi_{T_{n_k}}(x_{n_k})
		\longrightarrow S_q.
		\]
		Since $f$ decreases along the orbit, passage to the limit gives
		$c_{r_1}\ge c_q$.  Together with $c_{r_1}<c_p$ this implies that
		$c_q\le c_{r_1}<c_p$.  Hence
		\[
		\#\{r\mid c_q\le c_r\le c_{r_1}\}
		<
		\#\{r\mid c_q\le c_r\le c_p\}.
		\]
		The smallest possible cardinality for $\{r\mid c_q\le c_r\le c_p\}$ is two, since $p\ne q$. In that case there is no intermediate $r_1$, so the preceding case $r_1=q$ already finishes the proof. For larger cardinalities the inductive hypothesis yields $q<_F r_1$.  Together with
		$p\leadsto r_1$, this gives us $q<_F p$ as required.
	\end{proof}
	
	\begin{definition}\label{def:index-pair}
		A pair of compact subsets $E\subset N$ of $X$ is called an {\bf index pair} for a compact isolated invariant set $S$ if three axioms hold:
		\begin{enumerate}[label=\textup{(\roman*)}]
			\item $S \subset (N-E)^\circ$ and  $S = \Inv(\overline{N-E})$;
			\item $E$ is positively invariant\footnote{Explicitly, this means that $z \in E$ and $\phi_{[0,t]}(z) \subset N$ imply $\phi_{[0,t]}(z) \subset E$.} relative to $N$;
			\item every forward orbit leaving $N$ first meets $E$.
		\end{enumerate}
	\end{definition}
	
	Here $E$ is called the {\bf exit set} of the index pair (see \cite[\S III.4]{Conley1978}). Let $\Lambda$ be a commutative unital ring. The relative homology $H_*(N,E;\Lambda)$ is independent of the index pair and is denoted $\Con_*(S;\Lambda)$.  Similarly, the pointed homotopy type of the quotient $N/E$ is independent of the index pair; here and below we use the convention $N/\varnothing=N_+$. We denote it by $h(S)$ and call it the {\bf homotopy Conley index}
	\cite{Salamon1985, Mischaikow2002}.
	
	\subsection{Block Systems} Our goal here is to construct a specific family of index pairs for all the connected components of $\Crit(f)$.
	
	\begin{definition}\label{def:block-system}
		A {\bf block system} for the Morse decomposition \eqref{eq:morse-decomposition} is a number $\eps>0$ and
		pairwise disjoint compact subanalytic sets $\{N_p \mid p \in P\}$ such that
		\begin{enumerate}[label=\textup{(N\arabic*)},leftmargin=3.4em]
			\item $S_p\subset N_p^\circ$ and $N_p \subset f^{-1}[c_p-\eps~,~c_p+\eps]$;
			\item every $y\in\partial N_p$ with $f(y)>c_p-\eps$ satisfies
			$\phi_{(0,\delta)}(y)\subset N_p^\circ$ for some $\delta>0$.
		\end{enumerate}
	\end{definition}
	
	Condition (N2) asks that every boundary point of $N_p$ above the exit level be a strict entrance point, which is the defining feature of the isolating blocks of \cite{ConleyEaston1971}; this motivates our terminology. Fix a block system $\set{\epsilon, N_p}$ for the decomposition $\Crit(f) = \coprod_p S_p$ and for each $p \in P$ define the subset $E_p \subset N_p$ as
	\[
	E_p=N_p\cap f^{-1}(c_p-\eps).
	\]
	
	\begin{lemma}\label{lem:block-properties}
		The following properties hold for every $p\in P$:
		\begin{enumerate}[label=\textup{(\roman*)}]
			\item $S_p\subset N_p^\circ$ and $\Inv(N_p)=S_p$;
			\item $E_p$ is exactly the exit set of $N_p$;
			\item if $x\in E_p$ and $\phi_t(x)\in N_q$ for some $t>0$, then $q\prec_c p$;
			\item $(N_p,E_p)$ is an index pair for $S_p$;
			\item if $c_p=\min f$, then $E_p=\varnothing$.
		\end{enumerate}
	\end{lemma}
	
	\begin{proof}
		A critical point on $\partial N_p$ with value $c_p$ would violate Definition \ref{def:block-system} (N2), so $S_p\subset N_p^\circ$.  If a complete orbit lies in $N_p$, both asymptotic limits lie in $S_p$ by disjointness of the blocks, hence strict descent makes the orbit constant. Thus $S_p$ is the maximal invariant subset of the interior  $(N_p-E_p)^\circ$. If $x\in E_p$, then $f \circ \phi_t(x) <c_p-\eps$ for every $t>0$ by \eqref{eq:descent}, so the orbit leaves $N_p$ immediately.  Conversely, an exit point cannot have value greater than $c_p-\eps$ by Definition \ref{def:block-system} (N2), hence lies in $E_p$.  If an orbit from
		$E_p$ later meets $N_q$, then 
		\[
		c_q-\eps ~ \leq ~ f \circ \phi_t(x) ~ < ~ c_p-\eps,
		\]
		so $q\prec_c p$.  Since
		$S_p\subset N_p-E_p$ and
		$\Inv(\overline{N_p-E_p})\subset\Inv(N_p)=S_p$, the first index pair  axiom holds with the closure convention of \Cref{def:index-pair}; the other two axioms are immediate from the exit description.  Finally, if $c_p=\min f$, then $f^{-1}(c_p-\eps)=\varnothing$ and hence $E_p=\varnothing$.
	\end{proof}
	
	Isolating blocks always exist. We outline a concrete and convenient construction in the proof below, but note that other approaches are possible.
	
	\begin{theorem}\label{thm:blocks-exist}
		The decomposition $\Crit(f) = \coprod_{p \in P}S_p$ admits a block system.
	\end{theorem}
	
	\begin{proof}
		Let $\rho:X \to \R_{\geq 0}$ be the norm-squared of the gradient, i.e., $\rho(x)=\|\nabla f(x)\|^2$. Choose any $\mu>0$ such that $2\mu \cdot \sup_X\|\Hess f\|<1$, where $\Hess f$ denotes the Hessian of $f$. Then the function $\widehat f=f+\mu\rho$ strictly decreases along every nonconstant negative-gradient orbit:
		\[
		\frac{d}{dt}\widehat f(\phi_t(x))
		\leq
		-\left(1-2\mu\sup_{X}\|\Hess f\|\right)
		\|\nabla f(\phi_t(x))\|^2.
		\]
		Choose $\eps>0$ with $2\eps$ smaller than every nonzero critical-value gap. For each critical value $c$ of $f$, we define
		\[
		\mathscr A(c;\eps)
		:=
		\set{x\in X\mid f(x)\ge c-\eps \text{ and } \widehat f(x)\le c+\eps}.
		\]
		For fixed $c=c_p$, this gives a nested family of compact subanalytic sets which decrease to $\Crit(f) \cap f^{-1}(c)$ in the limit as $\eps\to 0$.  Choose pairwise disjoint open neighbourhoods $\set{V_i}$ of the components $\set{S_i}$ of $\Crit(f) \cap \set{f=c_p}$.  By compactness and the preceding nested intersection, for all sufficiently small $\eps > 0$ we have
		\[
		\mathscr A(c;\eps)\subset\bigcup_{c_p=c}V_p.
		\]
		Since each $S_p$ is connected and nonempty, there is a unique connected component $N_p \subset \mathscr A(c_p;\eps)$ which satisfies $S_p \subset N_p^\circ \subset V_p$.  The pairwise disjointness of the $V_p$ ensures that distinct critical components on the same level determine distinct blocks. Across different levels, disjointness is ensured because $2\epsilon < c_p - c_q$ holds whenever $c_p > c_q$, forcing $\mathscr A(c_p;\eps) \cap \mathscr A(c_q;\eps) = \varnothing$. 
		
		If $y\in\partial N_p$ and $f(y)>c_p-\eps$, then
		$\widehat f(y)=c_p+\eps$.  The point $y$ is noncritical for $f$, so it follows as above that $\widehat f$ decreases
		strictly under the forward flow, while $f$ remains larger than
		$c_p-\eps$ for short positive time.  The orbit therefore enters the interior
		of the same connected component of $\mathscr A(c_p;\eps)$ immediately.
	\end{proof}

	It will be useful in subsequent arguments to note that the $N_p$ which we have built in the preceding argument can be made arbitrarily tight around the $S_p$. By construction, we have
	\[
	\mathscr A(c;\eps)
	\subset
	\set{|f-c| \leq \eps} \cap
	\set{\|\nabla f\|^2\le 2\eps/\mu}.
	\]	
	Thus, by making $\epsilon$ sufficiently small, each $N_p$ is constrained to a small ball $B_\delta(S_p)$ around $S_p \subset \Crit(f)$ with respect to the Riemannian metric. 
	
	\begin{definition}\label{def:tight-blocks}
		A block system $\set{\epsilon,N_p}$ is \textbf{$\delta$-tight} if $N_p\subset B_\delta(S_p)$ holds for every $p \in P$.
	\end{definition}

	\subsection{Landfall}
	
	Fix a block system $\set{\epsilon,N_p}$ for \eqref{eq:morse-decomposition}. Write $c_p=f(S_p)$ and let $E_p=N_p\cap\set{f=c_p-\epsilon}$ be the exit set. For each $p\in P$, put $U_p^-:=\bigcup_{q\prec_c p}N_q^\circ$. Since $X$ is compact, every subanalytic subset of $X$, or of $[-T,T]\times X$, is globally subanalytic, i.e.\ definable in the o-minimal structure $\R_{\mathrm{an}}$ \cite{vanDenDries1998,Shiota1997}. We will use throughout that such sets are closed under finite Boolean operations, coordinate projections and fibrewise infima of bounded definable families. The {\bf landfall time} of $x \in E_p$ is defined as
	\begin{align}\label{eq:landfall-time}
		\tau_p(x) := \inf\set{t>0\mid\phi_t(x)\in U_p^-}.
	\end{align}
	For $x\in E_p$ the set over which this infimum is taken is nonempty by \Cref{lem:landfall} below. The subset $\phi_{[0,\tau_p(x)]}(x)  \subset X$ is known as the {\bf landfall transit} of $x \in E_p$. The following result catalogues basic properties of landfall times.
	
	\begin{lemma}\label{lem:landfall}
		If $E_p\neq\varnothing$, then:
		\begin{enumerate}[label=\textup{(\roman*)}]
			\item\label{landfall:positive} there is $\eta_p>0$ such that the intersection
			$\phi_{[0,\eta_p]}(E_p)\cap\bigcup_{q\prec_c p}N_q$ is empty;
			\item\label{landfall:finite} $\tau_p(x)$ is finite for every $x\in E_p$;
			\item\label{landfall:destination} there is a unique $q\prec_c p$ such that
			$\phi_{\tau_p(x)}(x)\in\partial N_q$ and the orbit enters the interior $N_q^\circ$
			for arbitrarily small times $t > \tau_p(x)$;
			\item\label{landfall:usc} $\tau_p$ is upper semicontinuous;
			\item \label{landfall:no-ret} if $x\in E_p$, then $\phi_t(x)\notin N_p$ for every $t>0$.
		\end{enumerate}
	\end{lemma}
	
	\begin{proof}
		For \ref{landfall:positive}, the compact disjoint sets
		$N_p$ and $\bigcup_{q\prec_c p}N_q$ have positive distance.  Uniform continuity of the flow on $[0,1]\times E_p$ therefore supplies $\eta_p>0$ such that no point of $E_p$ can reach a lower block in time at most $\eta_p$. Turning to \ref{landfall:finite}, the forward limit of $x$ lies in some lower critical set by
		\Cref{prop:capture,lem:block-properties}, so the orbit eventually enters a
		lower block.  Choose times decreasing to $\tau_p(x)$ at which the orbit lies
		in lower block interiors.  Finiteness of $P$ and closedness of the blocks
		place the terminal point on the boundary of a single lower block; disjointness
		gives uniqueness, and the chosen sequence gives arbitrarily small later times
		in that block interior. If $t>\tau_p(x)$ is chosen with
		$\phi_t(x)\in U_p^-$, every point in the open neighbourhood
		$\phi_t^{-1}(U_p^-)$ has landfall time at most $t$.  This proves \ref{landfall:usc}; the preceding limiting argument proves
		\ref{landfall:destination}. Finally,  \ref{landfall:no-ret} follows from assertion (iii) of \Cref{lem:block-properties}.
	\end{proof}
	
	The unique $q \in P$ guaranteed by \Cref{lem:landfall}\ref{landfall:destination} is called the {\bf landfall destination} of the given $x\in E_p$ and denoted $\dest_p(x)$. The following result establishes a local lower bound and a global upper bound on landfall times.	
	
	\begin{proposition}\label{prop:uniform-transit}
		The maximal $\phi$-invariant subset of $G := X-\bigcup_{p\in P}N_p^\circ$ is empty. There is $T_G<\infty$ such that every orbit segment contained in $G$ has length less than $T_G$; moreover, every landfall transit satisfies
		$\phi_{[0,\tau_p(x)]}(x)\subset G$ and $
		\eta_p\le\tau_p(x)<T_G$ whenever $E_p\neq\varnothing$, where $\eta_p >0$   is as in Lemma \ref{lem:landfall}(i).
	\end{proposition}
	
	\begin{proof}
		A complete orbit in $G$ would have a forward limit in some
		$S_q\subset N_q^\circ$, a contradiction.  For $n\ge1$ consider
		\[
		G_{(n)} := \set{x\in G\mid \phi_{[-n,n]}(x)\subset G}.
		\]
		These compact sets are nested with empty intersection, so $G_{(n_0)}=\varnothing$ for some $n_0$; take $T_G=2n_0$.  If an orbit segment of length at least $T_G$ were contained in $G$, its midpoint would belong to $G_{(n_0)}$, a contradiction. Before time $\tau_p(x)$ the orbit cannot meet a block interior: any such block must have lower critical value by \Cref{lem:block-properties}(iii), contradicting the definition of landfall.  The terminal point lies on a block boundary by \Cref{lem:landfall}.  Thus the entire landfall transit lies in $G$, giving the upper bound.  The lower bound is \Cref{lem:landfall}\ref{landfall:positive}.
	\end{proof}
	
	A significant benefit of the finite bound $T_G$ established above is that the landfall time functions inherit subanalyticity.
	
	\begin{lemma}\label{lem:landfall-subanalytic}
		For every $p \in P$, the landfall time function $\tau_p:E_p\to\R$ is subanalytic.
	\end{lemma}
	
	\begin{proof}
		The assertion is vacuous for $E_p=\varnothing$; otherwise, for each $q\prec_c p$ define
		\[
		A_{qp}=
		\set{(x,t)\in E_p\times(0,T_G]\mid \phi_t(x)\in N_q^\circ}.
		\]
		This set is subanalytic since we only flow for bounded time. For each $x \in E_p$ we let $A_{qp,x} \subset (0,T_G]$ be the set of all $t$ such that $(x,t) \in A_{qp}$; each such set is a coordinate projection of $A_{qp}$ and hence subanalytic. Thus, the function $u_{qp}:E_p \to (0,T_G+1]$ given by
		\[
		u_{qp}(x) = \begin{cases}
			\inf A_{qp,x} & A_{qp,x} \neq \varnothing \\
			T_G + 1 & \text{otherwise}
		\end{cases}
		\] is subanalytic.  Since $P$ is finite and every transit lands before $T_G$ by \Cref{prop:uniform-transit}, we have
		$\tau_p=\min_{q\prec_c p}u_{qp}$. Since the minimum is indexed over a finite set, the conclusion follows.
	\end{proof}
	
	We now decompose each exit set $E_p$ based on the collection of isolating blocks $N_q$, indexed over $q \prec_c p$, encountered by the positive orbits prior to landfall. 
	
	\begin{definition}\label{def:contact}
		Given $q\prec_c p$ in $P$, a point $x\in E_p$ is said to {\bf contact} $N_q$ if there exists
		$t\in(0,\tau_p(x)]$ such that $\phi_t(x)\in\partial N_q$; the {\bf contact domain} $\mathcal C_{qp}\subset E_p$ consists of all $x \in E_p$ which contact $N_q$.
	\end{definition}
	
	Note that requiring $t \leq \tau_p(x)$ in the preceding definition forces $\phi_s(x)\notin N_r^\circ$ for every $s \in (0,t)$ and every $r \in P$: blocks with $r\prec_c p$ are excluded by the definition \eqref{eq:landfall-time} of $\tau_p$, and all others by \Cref{lem:block-properties}(iii).

	\begin{lemma}\label{lem:contact-branches}
		For every $q\prec_c p$:
		\begin{enumerate}[label=\textup{(\roman*)}]
			\item each $x\in\mathcal C_{qp}$ has a unique contact time
			$\tau_{qp}(x) \in [\eta_p,T_G)$ such that $\phi_{\tau_{qp}(x)}(x) \in \partial N_q$;
			\item $\mathcal C_{qp}$ is compact and subanalytic, and
			$\tau_{qp}:\mathcal C_{qp} \to \R$ is continuous and subanalytic;
			\item the map $x \mapsto \phi_{\tau_{qp}(x)}(x)$ 
			is a continuous subanalytic embedding
			$\chi_{qp}:\mathcal C_{qp}\to\partial N_q$.
		\end{enumerate}
	\end{lemma}
	
	\begin{proof}
		By \Cref{lem:landfall}\ref{landfall:positive}, every contact time is at
		least $\eta_p$.  Consider
		the contact pairs in $E_p\times[\eta_p,T_G]$.  They form a closed subset of
		the hypograph $t\le\tau_p(x)$, since the contact point condition is closed; this hypograph is compact because $\tau_p$ is upper semicontinuous and bounded (\Cref{lem:landfall}\ref{landfall:usc} and \Cref{prop:uniform-transit}).  The set is subanalytic by bounded quantification, using
		\Cref{lem:landfall-subanalytic} and bounded-time subanalyticity of the flow.
		Hence it is compact.  Suppose
		$t_1<t_2$ are two $q$-contact times for the same $x$.  At the first contact
		the orbit either enters $N_q^\circ$ immediately, which would force $\tau_p(x)\leq t_1<t_2$, or lies in $E_q$, in which case \Cref{lem:landfall}\ref{landfall:no-ret}
		rules out the second contact.  Thus the contact time is unique.  The compact
		graph theorem gives continuity, while the preceding description identifies the
		contact-pair set with the subanalytic graph of $\tau_{qp}$.  Finally two exit points with the same contact point lie on the
		same orbit and on the same regular level of $f$, hence coincide.
	\end{proof}
	
	\noindent The following result provides a truncation law for the contact maps $\chi_{qp}$.
	
	\begin{theorem}\label{thm:factorisation}
		Let $x\in\mathcal C_{rp}$ satisfy $\tau_{rp}(x)<\tau_p(x)$, 
		and set $y=\chi_{rp}(x)$. Then:
		\begin{enumerate}[label=\textup{(\roman*)}]
			\item we have $y\in E_r$ and
			\begin{align}\label{eq:tail-landfall}
				\tau_r(y)=\tau_p(x)-\tau_{rp}(x).
			\end{align}
			\item for every $q \prec_c r$,
			\[
			y\in\mathcal C_{qr}
			\quad\Longleftrightarrow\quad
			x\in\mathcal C_{qp}\ \text{ and }\ 
			\tau_{qp}(x)>\tau_{rp}(x);
			\]
			\item if $y \in \mathcal C_{qr}$ for $q \prec_c r$, then we have the truncation law
			\begin{align}\label{eq:contact-factorisation}
				\tau_{qr}(y) = \tau_{qp}(x)-\tau_{rp}(x), \qquad
				\chi_{qr}(y)=\chi_{qp}(x).
			\end{align}
		\end{enumerate}
	\end{theorem}
	
	\begin{proof}
		Because the $r$-contact is nonterminal, Definition \ref{def:block-system} (N2) forces $y\in E_r$: otherwise the forward orbit would enter $N_r^\circ$ immediately, before time $\tau_p(x)$.  After time $\tau_{rp}(x)$ the original orbit is the orbit from $y$.  By definition of $\tau_p(x)$ it enters no lower block interior before its terminal time, while at that terminal time it enters the interior of the actual destination block.  This proves \eqref{eq:tail-landfall}.
		
		If $y\in\mathcal C_{qr}$ at time $u$, then the original orbit contacts $N_q$ at time $\tau_{rp}(x)+u$.  There cannot be an earlier $q$-contact before the $r$-contact: such a contact would be nonterminal and would lie in $E_q$, and the later $r$-contact would then contradict \Cref{lem:block-properties}(iii) applied to the orbit from $E_q$, since $c_q<c_r$.  Hence this is the unique $q$-contact of the original transit.
		
		Conversely, if the original transit contacts $N_q$ after the $r$-contact, then that contact is also a contact of the tail from $y$.  Writing $t_r=\tau_{rp}(x)$, the flow identity gives
		\[
		\phi_u(y)=\phi_{u+t_r}(x)
		\qquad (u\ge0).
		\]
		Uniqueness of the contact times therefore gives \eqref{eq:contact-factorisation}.  Finally, if
		two contacts occur at times $t_r<t_q$, then the first is nonterminal and the
		preceding argument gives $c_q<c_r$.  Distinct contacts cannot occur at the same
		time because the blocks are disjoint.
	\end{proof}
	
	\begin{corollary}\label{cor:contact-order}
		If a $p$-transit contacts both $N_r$ and $N_q$ with $c_q<c_r<c_p$, then $\tau_{rp}(x)<\tau_{qp}(x)$.
		In particular the earlier of two distinct contacts always has the larger critical value, and every nonterminal contact lies in the corresponding exit set.
	\end{corollary}
	
	\begin{proof}
		If the $q$-contact occurred first, it would be nonterminal and hence would lie
		in $E_q$ by \Cref{thm:factorisation}.  The later $r$-contact would then
		contradict \Cref{lem:block-properties}(iii) for the orbit issuing from $E_q$.
		Distinct contacts cannot occur at the same time because the blocks are
		disjoint.  The final statement about nonterminal contacts is the first
		conclusion of \Cref{thm:factorisation}.
	\end{proof}
	
	Easton's symbolic itineraries record ordered block crossings in greater dynamical generality \cite{Easton1975}.  What is special in our setting here is the truncation law
	\eqref{eq:contact-factorisation}. 
	
	\begin{proposition}\label{prop:chambers}
		For every $p \in P$ there is a partition
		\[
		E_p=\coprod_{q\prec_c p}E_{qp},
		\]
		where $E_{qp} := \set{x\in\mathcal C_{qp}\mid\tau_{qp}(x)=\tau_p(x)}$. 
		The restriction of $\tau_p$ to $E_{qp}$ is continuous, and we have $\chi_{qp}(E_{qp})\subset\partial N_q-E_q$. For every $x\in (\overline E_{qp}-E_{qp})$, we have
		$\tau_{qp}(x)<\tau_p(x)$ and $\chi_{qp}(x)\in E_q$.
	\end{proposition}
	
	\begin{proof}
		The terminal landfall time satisfies the $q$-contact condition for the unique
		destination $q$, and uniqueness of the contact time gives the displayed
		description.  If a terminal contact lay in $E_q$, \Cref{lem:landfall}\ref{landfall:no-ret} would
		contradict immediate entry into $N_q^\circ$.  On the other hand a
		nonterminal point in the closure cannot lie in the entrance part of
		$\partial N_q$, since Definition \ref{def:block-system} (N2) would give entry into $N_q^\circ$ before
		landfall.
	\end{proof}
	
	\noindent We call $E_{qp}$ the $q$-{\bf chamber} of $E_p$. A point of $E_{qp}$ which lies in the closure of a different chamber $E_{rp}$ is said to lie on a {\bf landfall wall} of $E_p$; the next result shows that crossing such a wall factors the terminal contact through an intermediate block.
	
	\begin{corollary}\label{cor:wall}
		Given $q\neq r$ in $P$, assume the existence of a sequence $\set{x_n} \subset E_{rp}$ converging to $x\in E_{qp}$. Then,
		\[
		q\prec_c r\prec_c p
		\quad \text{and} \quad
		\tau_p(x_n)\to\tau_{rp}(x)<\tau_p(x).
		\]
		Setting $y=\chi_{rp}(x)$, one has
		\[
		y\in E_{qr},
		\qquad
		\tau_r(y)=\tau_p(x)-\tau_{rp}(x),
		\qquad
		\chi_{qr}(y)=\chi_{qp}(x).
		\]
	\end{corollary}
	
	\begin{proof}
		Since $\mathcal C_{rp}$ is compact, its limit $x$ lies in
		$\mathcal C_{rp}$.  On $E_{rp}$ we have
		$\tau_p=\tau_{rp}$, so continuity of $\tau_{rp}$ yields $\tau_p(x_n)=\tau_{rp}(x_n) \to \tau_{rp}(x)$. The $r$-contact of $x$ is nonterminal because $x$ has destination
		$q \neq r$, hence $\tau_{rp}(x)<\tau_p(x)$.  Applying \Cref{thm:factorisation} at this nonterminal contact gives $y=\chi_{rp}(x)\in E_r$ and $\tau_r(y)=\tau_p(x)-\tau_{rp}(x)$. The terminal $q$-contact of the original transit becomes the terminal $q$-contact of the tail from $y$, so $y\in E_{qr}$ and 
		$\chi_{qr}(y)=\chi_{qp}(x)$.  Finally
		\Cref{lem:block-properties}(iii) applied first to the $p$-transit and then
		to the tail from $E_r$ gives $q\prec_c r\prec_c p$.
	\end{proof}
	
	\subsection{Contact Domains and the Flow Ordering}\label{subsec:tight-flow-order}
	
	As observed after \Cref{thm:blocks-exist}, for every $\delta>0$ the blocks constructed there are $\delta$-tight once the parameter $\eps>0$ is sufficiently small. Our goal here is to show that for sufficiently small $\delta$, the nonemptiness of the contact domains induced by a $\delta$-tight block system allows us to reconstruct the flow order $<_F$ on $P$ from \Cref{def:flow-order}. This order depends only on the gradient flow and is not a property of the chosen block systems.
	
	\begin{theorem}\label{thm:contact-support}
		There is $\delta_0>0$ such that every $\delta_0$-tight block system satisfies
		\[
		\mathcal C_{qp}\ne\varnothing
		\quad\Longrightarrow\quad
		q<_F p.
		\]
	\end{theorem}
	
	\begin{proof}
		We proceed by contradiction. Assume otherwise, so that for every $n$ there is a $1/n$-tight block system and a pair $(q_n,p_n)$ whose contact domain is nonempty but with
		$q_n\not<_F p_n$.  Since $P$ is finite, we may pass to a subsequence on which the pair is constant, say $(q,p)$.  Let
		$\mathcal C^{n}_{qp}$, $\tau^{n}_{qp}$ and $\chi^{n}_{qp}$ be the contact data for the $n$-th block system.  Choose $x_n\in\mathcal C^{n}_{qp}$, set
		$T_n=\tau^{n}_{qp}(x_n)$, and 
		$y_n=\chi^{n}_{qp}(x_n)$.  Tightness gives
		\[
		\dist_g(x_n,S_p)< n^{-1},
		\qquad
		\dist_g(y_n,S_q)< n^{-1},
		\]
		which contradicts \Cref{lem:broken-flow-order}.
	\end{proof}
	
	\begin{proposition}\label{prop:genuine-connections-survive}
		There is $\delta_1\in(0,\delta_0]$ such that every $\delta_1$-tight block
		system satisfies
		\[
		p\leadsto q
		\quad\Longrightarrow\quad
		E_{qp}\ne\varnothing.
		\]
	\end{proposition}
	
	\begin{proof}
		For each pair with $p\leadsto q$, choose one witness orbit $\gamma_{qp}$.
		Its closure consists of the orbit together with its two asymptotic limits and
		contains no critical point outside $S_p\cup S_q$.  Hence
		\[
		\dist\bigl(S_r,\overline{\gamma_{qp}}\bigr)>0
		\qquad \text{for all } r \in P - \set{p,q}.
		\]
		There are only finitely many chosen witnesses, so choose
		$\delta_1\le\delta_0$ smaller than all these positive distances. For a $\delta_1$-tight system, $\gamma_{qp}$ misses every block other than
		$N_p$ and $N_q$.  Since its backward limit lies in
		$S_p\subset N_p^\circ$, the orbit eventually leaves $N_p$; its first exit
		point $x$ lies in $E_p$ by \Cref{lem:block-properties}(ii), and
		\Cref{lem:landfall}\ref{landfall:no-ret} prevents a return to $N_p$.  Its forward limit lies in
		$S_q\subset N_q^\circ$, so the forward orbit from $x$ eventually enters
		$N_q$ and meets no other block first.  Thus $\dest_p(x)=q$, whence
		$x\in E_{qp}$.
	\end{proof}
	
	\begin{corollary}\label{cor:support-sandwich}
		For a $\delta_1$-tight block system, we have
		\[
		p \leadsto q
		\implies 
		\mathcal C_{qp} \neq \varnothing \implies q
		<_F p
		\]
		for every $p \neq q$ in $P$. Consequently the transitive closure of $\set{(q,p) \mid \mathcal{C}_{qp} \neq \varnothing}$ is $<_F$. 
	\end{corollary}
	
	\begin{proof}
		The two inclusions are \Cref{prop:genuine-connections-survive,thm:contact-support};
		the first uses $E_{qp}\subset\mathcal C_{qp}$.  Taking transitive closures
		gives the equality because $<_F$ is, by definition, the transitive closure of the relation
		$\leadsto$. 
	\end{proof}
	
	Henceforth every block system is assumed to be $\delta_1$-tight, and we call such block systems {\bf tight}. Tight block systems exist: by the remarks following \Cref{thm:blocks-exist}, the construction given there is $\delta_1$-tight once $\eps>0$ is sufficiently small.

	\section{The Landfall Category}\label{sec:category}
	
	Here we assemble the contact data of \Cref{sec:blocks} into a topological category, construct an evaluation map from its classifying space to $X$, and show that this map is a simple homotopy equivalence compatible with the down-set filtrations; this proves Theorem (A). Throughout, fix a tight block system $\set{\epsilon,N_p}_{p\in P}$ for \eqref{eq:morse-decomposition}, equip $P$
	with its flow order, and write $E_p=N_p\cap\{f=c_p-\epsilon\}$ for the exit set. We denote the affiliated contact domains as $\mathcal C_{qp}$ and contact maps as $\chi_{qp}$ --- see \Cref{lem:contact-branches}.
	
	\subsection{Basic Definition}
	
	The truncation law from \Cref{thm:factorisation} assembles the contact data into a topological category.
	
	\begin{definition}\label{def:landfall-category}
		The \textbf{landfall category} $\cL$ associated to the block system $\set{\epsilon,N_p}$ is the strict nonunital topological category defined as follows:
		\begin{enumerate}[label=\textup{(\roman*)}]
			\item the object space is $\Ob(\cL)=\coprod_{p\in P}N_p$,
			\item the morphism space is $\Mor(\cL)=\coprod_{q <_F p}\mathcal C_{qp}$,
			\item the source and target of $x\in\mathcal C_{qp}$ are $s(x)=x$ and $t(x)=\chi_{qp}(x)$,
			\item the composite of $x\in\mathcal C_{qp}$ with $y=\chi_{qp}(x)\in\mathcal C_{rq}$ is $x$, now regarded as a point in $\mathcal C_{rp}$.
		\end{enumerate}
	\end{definition}
	
	\noindent Our first order of business is to show that composition is sensible as defined.
	
	\begin{proposition}\label{prop:category}
		The composition law of $\cL$ is well defined, continuous, and strictly associative.
	\end{proposition}
	
	\begin{proof}
		Suppose that $x\in\mathcal C_{qp}$ and
		$y=\chi_{qp}(x)\in\mathcal C_{rq}$. Since $y\in E_q$, the $q$-contact of the $p$-transit is nonterminal. The converse direction of
		\Cref{thm:factorisation} therefore gives $x\in\mathcal C_{rp}$ and
		$\chi_{rp}(x)=\chi_{rq}(y)$, whence composition is well defined. On each space of composable pairs, composition is projection to the original source point $x$, and is therefore continuous. Associativity is the statement that deleting two intermediate contacts from one ordered contact list is independent of the order in which they are deleted; this is the truncation law of \Cref{thm:factorisation}.
	\end{proof}
	
	Every nonidentity contact strictly lowers the flow order, and hence $\cL$ has no identity morphisms. Its nerve is therefore naturally a \emph{semi-simplicial} space $N_\bullet\cL$, as described in \cite{EbertRandalWilliams2019} for instance. In particular, the space of $0$-simplices is
	$N_0\cL=\Ob(\cL)$, while $N_k\cL$ for $k\geq1$ corresponds to the space of $k$-tuples of composable morphisms. We write $B\cL:=|N_\bullet\cL|$ for the geometric realisation. For each chain $\mathbf p=(p_0>_F p_1>_F \cdots >_F p_k)$
	in $P$, let $N_k\cL(\mathbf p)$ denote the corresponding union of components of $N_k\cL$.
	
	\begin{proposition}\label{prop:segal}
		There is a canonical homeomorphism
		\[
		N_k\cL(\mathbf p)
		\cong
		\bigcap_{i=1}^k\mathcal C_{p_ip_0}
		\]
		for every chain $\mathbf p=(p_0 >_F p_1 >_F \cdots>_F p_k)$ with $k \ge 1$. Under this identification, the face map $d_i$ for $i>0$ forgets the $p_i$-contact, while $d_0(x)=\chi_{p_1p_0}(x)$.
	\end{proposition}
	
	\begin{proof}
		For $k=1$ there is nothing to prove. For $k=2$, the map
		\[
		\mathcal C_{p_1p_0}\cap\mathcal C_{p_2p_0}
		\longrightarrow
		\mathcal C_{p_2p_1}\times_{\Ob(\cL)}\mathcal C_{p_1p_0},
		\qquad
		x\longmapsto(\chi_{p_1p_0}(x),x),
		\]
		is well defined by
		\Cref{cor:contact-order,thm:factorisation}: the $p_1$-contact occurs before the $p_2$-contact and is therefore nonterminal. Conversely, \Cref{thm:factorisation} shows that every composable pair arises in this way. The map is injective by its second coordinate, and hence is a homeomorphism because its source is compact and its target Hausdorff. Iterating the same truncation argument gives the result for arbitrary $k$. The description of the face maps follows immediately: composing two consecutive morphisms deletes the corresponding intermediate contact, while $d_0$ truncates the transit at its first listed contact.
	\end{proof}
	
	There is also a natural structure map from the landfall category to the flow-order poset. On geometric realisations it gives
	\[
	\Gamma:B\cL \to \Kord(P),
	\]
	where $\Kord(P)$ is the order complex of $P$: on the nerve component indexed by $p_0 >_F \cdots >_F p_k$, the map $\Gamma$ simply retains the barycentric coordinates and forgets the contact point.
	
	\subsection{Transit Spaces and Evaluation}
	
	For each $p\in P$, we define the \textbf{transit space}
	\[
	\widetilde\Theta_p=\set{(x,t)\mid x\in E_p \text{ and }0\leq t\leq\tau_p(x)},
	\]
	and let $\ev_p:\widetilde\Theta_p\to X$ be the map $(x,t)\mapsto\phi_t(x)$. Let $\Theta_p$ be the image $\ev_p(\widetilde\Theta_p)$.
	
	\begin{lemma}\label{lem:transit-space}
		Both $\widetilde\Theta_p$ and $\Theta_p$ are compact subanalytic sets, the map $\ev_p:\widetilde\Theta_p\to\Theta_p$ is a homeomorphism, and $\Theta_p\cap N_p=E_p$.
	\end{lemma}
	
	\begin{proof}
		The hypograph of the bounded upper-semicontinuous subanalytic function $\tau_p$ is compact and subanalytic by
		\Cref{lem:landfall,prop:uniform-transit,lem:landfall-subanalytic}. Hence $\widetilde\Theta_p$ is compact subanalytic.
		If $\phi_t(x)=\phi_{t'}(x')$ for $x,x'\in E_p$, then $x$ and $x'$ lie on the same nonconstant orbit and share their $f$-value $c_p-\epsilon$, whence $x=x'$. Uniqueness of the flow then gives $t=t'$. Thus $\ev_p$ is injective. Its image is compact subanalytic because the flow is subanalytic on the bounded time interval $[0,T_G]$. Continuous bijections from compact spaces to Hausdorff spaces are homeomorphisms onto their images. Finally, the combination of $f(\phi_t(x))<c_p-\epsilon$ for every $x\in E_p$ and $t>0$ with $N_p\subset\{f\geq c_p-\epsilon\}$ forces $\Theta_p\cap N_p=E_p$.
	\end{proof}
	
	Define $\widehat N_p:=N_p\cup\Theta_p$ for each $p\in P$. We recall that $A\subset P$ is a {\bf down-set} if $q<_F p$ and $p\in A$ force $q\in A$. For each such $A$, define
	\[
	X_A:=\bigcup_{p\in A}\widehat N_p.
	\]
	
	\begin{proposition}\label{prop:transit-lattice}
		For every down-set $A\subset P$:
		\begin{enumerate}[label=\textup{(\roman*)}]
			\item $X_A$ is compact, subanalytic, positively invariant, and satisfies
			\[
			X_A=\phi_{[0,\infty)}\left(\bigcup_{p\in A}N_p\right).
			\]
			In particular, $X_P=X$. We call $X_A$ the {\bf forward saturation} of $\bigcup_{p\in A}N_p$.
			
			\item If $p\ne q$ and $\widehat N_p\cap\widehat N_q\ne\varnothing$, then $p$ and $q$ are comparable in the flow order.
			
			\item For any other down-set $B\subset P$, we have
			\[
			X_{A\cup B}=X_A\cup X_B,
			\qquad
			X_{A\cap B}=X_A\cap X_B.
			\]
		\end{enumerate}
	\end{proposition}
	
	\begin{proof}
		Compactness and subanalyticity follow from \Cref{lem:transit-space} and finiteness of the unions. Starting from $N_p$, a forward orbit either remains in that block for all future time or leaves through $E_p$, follows its $p$-transit, and lands in a block $N_q$ for which $\mathcal C_{qp}\ne\varnothing$. Hence $q <_F p$ by \Cref{thm:contact-support}. Iterating this procedure strictly lowers the finite flow order, and therefore remains inside $A$ whenever $p\in A$. Thus
		\[
		\phi_{[0,\infty)}\left(\bigcup_{p\in A}N_p\right)\subset X_A,
		\]
		and hence $X_A$ is positively invariant. The reverse inclusion follows directly from the definition of transit spaces. Conversely, every point of $X$ has a backward limit in some $S_p\subset N_p^\circ$ by \Cref{prop:capture}, and therefore lies in the forward saturation of $N_p$. Hence $X_P=X$.
		
		For (ii), recall that the blocks $\set{N_p}$ are pairwise disjoint by \Cref{def:block-system}. Since no transit may enter a block interior before terminal landfall, we must have $N_p\cap\Theta_q \subset \partial N_p$. Write any $z\in N_p\cap\Theta_q$ as $z=\phi_t(y)$ with $y\in E_q$ and $0\leq t\leq\tau_q(y)$. This is a $p$-contact of the $q$-transit, so $y\in\mathcal C_{pq}$ and $z=\chi_{pq}(y)$, forcing $p <_F q$. The case $z\in\Theta_p\cap N_q$ is symmetric. Finally, suppose $z\in\Theta_p\cap\Theta_q$, so
		$z=\phi_t(x)=\phi_s(y)$ for unique $x\in E_p$ and $y\in E_q$. If $c_p=c_q$, then $x$ and $y$ lie on the same orbit and on the same noncritical $f$-level, forcing $x=y$ and contradicting disjointness of the blocks. Otherwise, assume $c_p>c_q$. Then $x$ occurs before $y$ on their common orbit. Since $z$ occurs before or at the terminal time of the $p$-transit, that transit contacts $N_q$ at $y$. Hence $\mathcal C_{qp}\ne\varnothing$, so $q<_F p$.
		
		The union identity in (iii) is immediate. If $z\in X_A\cap X_B$, select $p\in A$ and $q\in B$ such that $z\in\widehat N_p\cap\widehat N_q$. If $p=q$, then $p\in A\cap B$. Otherwise, (ii) forces $p$ and $q$ to be comparable; assume $q<_F p$. By downward closure, $q\in A$, whence $z\in\widehat N_q\subset X_{A\cap B}$.
	\end{proof}
	
	Consider a descending chain $\mathbf p=(p_0>_F p_1>_F\cdots>_F p_k)$ in $P$ and a point
	$x\in\bigcap_{i=1}^k\mathcal C_{p_ip_0}$. Let $t^x\in\R^{k+1}$ be the vector with $t_0^x=0$ and $t_i^x=\tau_{p_ip_0}(x)$ for $1\leq i\leq k$. For $\lambda=(\lambda_0,\ldots,\lambda_k)\in\simp^k$, define
	\begin{equation}\label{eq:evaluation}
		\Ev(x;\lambda)=\phi_{\ip{\lambda,t^x}}(x).
	\end{equation}
	This map extends to $N_0\cL=\coprod_pN_p$ via $\Ev(z)=z$.
	
	\begin{lemma}\label{lem:evaluation-well-defined}
		Evaluation induces a continuous map $\Ev:B\cL\to X$.
	\end{lemma}
	
	\begin{proof}
		On each component of $N_k\cL\times\simp^k$, the formula is evidently continuous, so it suffices to check compatibility with the face identifications in the geometric realisation. For $i>0$, the face map $d_i$ merely deletes the $p_i$-contact by \Cref{prop:segal}; on the corresponding simplex face $\lambda_i=0$, the evaluation formula is therefore unchanged. For $d_0$, set $y=\chi_{p_1p_0}(x)=\phi_{t_1^x}(x)$, which lies in the nerve summand indexed by $(p_1>_F\cdots>_F p_k)$. By \Cref{thm:factorisation}, one has $t_{i-1}^y=t_i^x-t_1^x$ for $1\leq i\leq k$. On the face $\lambda_0=0$,
		\[
		\phi_{\sum_{i=1}^k\lambda_it_i^x}(x)
		=
		\phi_{\sum_{i=2}^k\lambda_i(t_i^x-t_1^x)}(y),
		\]
		because $\sum_{i=1}^k\lambda_i=1$. Thus evaluation respects every face identification and descends continuously to $B\cL$.
	\end{proof}
	
	\subsection{Triangulations and Flow-Prisms}
	
	We now select finite triangulations of the blocks $N_p$ which are compatible with the morphism spaces of $\cL$.
	
	\begin{lemma}\label{lem:triangulations}
		There exist finite subanalytic triangulations $\cK_p$ of the blocks $N_p$ such that:
		\begin{enumerate}[label=\textup{(\roman*)}]
			\item $E_p$ and every contact domain $\mathcal C_{qp}$ are subcomplexes of $\cK_p$;
			\item every $E_{qp}$ is a union of simplex interiors, and $\overline E_{qp}$ is a subcomplex;
			\item the image $\chi_{qp}(\sigma)$ of every simplex $\sigma\subset\mathcal C_{qp}$ is a subcomplex of $\cK_q$.
		\end{enumerate}
	\end{lemma}
	
	\begin{proof}
		All sets involved are subanalytic by
		\Cref{lem:contact-branches,prop:chambers}; the contact domains and chamber closures are compact, while the chambers themselves are locally closed. Fix a linear extension of the flow order and construct the triangulations from higher to lower indices. When $\cK_q$ is chosen, all $\cK_p$ with $p>_F q$ have already been fixed. Apply the subanalytic triangulation theorem \cite{Shiota1988Triangulation} to a finite family containing $E_q$, all contact domains and chamber closures inside $N_q$, all chambers $E_{rq}$, and all compact subanalytic sets $\chi_{qp}(\sigma)$, where $p>_F q$ and $\sigma$ ranges over the simplices of $\cK_p$ contained in $\mathcal C_{qp}$. Compatibility with the locally closed chambers makes them unions of simplex interiors, while compatibility with the compact members of the family makes them subcomplexes.
	\end{proof}
	
	A collection $\set{\cK_p\mid p\in P}$ satisfying the preceding lemma is called a {\bf flow-subordinate triangulation}. Fix such a collection once and for all. For every simplex $\sigma\subset E_p$, define
	\[
	R(\sigma):=\set{q<_F p\mid \sigma\subset\mathcal C_{qp}};
	\]
	we adopt the convention that $R(\sigma)=\varnothing$ for simplices $\sigma\not\subset E_p$.
	
	\begin{lemma}\label{lem:simplex-contacts}
		For every simplex $\sigma\subset E_p$:
		\begin{enumerate}[label=\textup{(\roman*)}]
			\item $R(\sigma)$ is nonempty and totally ordered;
			\item the minimum element of $R(\sigma)$ is the unique $q$ satisfying $\sigma^\circ\subset E_{qp}$;
			\item if $\rho\leq\sigma$ is a face, then $R(\sigma)\subset R(\rho)$.
		\end{enumerate}
	\end{lemma}
	
	\begin{proof}
		By \Cref{lem:triangulations}(ii) and the chamber partition of \Cref{prop:chambers}, the relative interior $\sigma^\circ$ lies in a unique chamber $E_{qp}$. Since $E_{qp}\subset\mathcal C_{qp}$ and $\mathcal C_{qp}$ is a subcomplex of $\cK_p$, we have $\sigma\subset\mathcal C_{qp}$ and hence $q\in R(\sigma)$. Choose $x\in\sigma^\circ$. Every element of $R(\sigma)$ is a contact label of the single $p$-transit through $x$. Distinct contacts are strictly ordered by contact time, and \Cref{thm:factorisation,thm:contact-support} show that the later contact has the lower flow label. Therefore $R(\sigma)$ is totally ordered. Since the $q$-contact is terminal for $x$, it occurs after every other contact, and hence $q=\min R(\sigma)$. Finally, each $\mathcal C_{qp}$ is a subcomplex, so $\sigma\subset\mathcal C_{qp}$ implies $\rho\subset\mathcal C_{qp}$ for every face $\rho\leq\sigma$.
	\end{proof}
	
	Let $\dest(\sigma):=\min R(\sigma)$ for every simplex $\sigma\subset E_p$. Thus $\dest(\sigma)$ is the common landfall destination of the points in $\sigma^\circ$, and $\rho\leq\sigma$ implies $\dest(\rho)\leq_F\dest(\sigma)$.
	
	\begin{definition} \label{def:flow-prism}
		The \textbf{flow prism} of a simplex $\sigma\subset E_p$ with $\dest(\sigma)=q$ is
		\[
		\Pi_\sigma:=\set{\phi_t(x)\mid x\in\sigma,\ 0\leq t\leq\tau_{qp}(x)}.
		\]
	\end{definition}
	
	\begin{lemma}\label{lem:flow-prisms}
		For every simplex $\sigma\subset E_p$ with $\dest(\sigma)=q$, the map $\Psi_\sigma:\sigma\times[0,1]\to\Pi_\sigma$ given by
		$(x,s) \mapsto \phi_{s\tau_{qp}(x)}(x)$ is a subanalytic homeomorphism. Moreover,
		$\Theta_p=\bigcup_{\sigma\subset E_p}\Pi_\sigma$.
	\end{lemma}
	
	\begin{proof}
		Since $\sigma\subset\mathcal C_{qp}$, \Cref{lem:contact-branches} shows that $\tau_{qp}$ is continuous and subanalytic on $\sigma$, with $\tau_{qp}(x)\geq\eta_p>0$. Thus $\Psi_\sigma$ is continuous and subanalytic. If
		$\phi_{s\tau_{qp}(x)}(x)=\phi_{s'\tau_{qp}(x')}(x')$, then $x$ and $x'$ lie on the same nonconstant orbit and on the common level $\{f=c_p-\epsilon\}$. Strict descent gives $x=x'$, and uniqueness of the flow then gives $s\tau_{qp}(x)=s'\tau_{qp}(x)$. Since $\tau_{qp}(x)>0$, we have $s=s'$. Thus $\Psi_\sigma$ is injective and, by compactness, a homeomorphism onto its image. For the final assertion, let $\sigma\subset E_p$ be the unique simplex whose relative interior contains $x$. By \Cref{lem:simplex-contacts}, the destination of $x$ is $\dest(\sigma)$, so $\tau_p(x)=\tau_{\dest(\sigma),p}(x)$. Thus the entire landfall transit of $x$ lies in $\Pi_\sigma$. The reverse inclusion is immediate.
	\end{proof}
	
	Flow prisms allow us to choose a triangulation of $X$ compatible with the blocks and their transits.
	
	\begin{proposition}\label{prop:global-triangulation}
		There is a finite subanalytic triangulation of $X$ compatible with every block $N_p$ and every flow prism $\Pi_\sigma$. Consequently every $\Theta_p$, every $\widehat N_p$, and every $X_A$ is a subpolyhedron. In particular, if $A'\subset A$ are down-sets, then $X_{A'}\hookrightarrow X_A$ is a cofibration.
	\end{proposition}
	
	\begin{proof}
		There are finitely many blocks and finitely many simplices in the flow-subordinate triangulations. By \Cref{lem:flow-prisms}, every $\Pi_\sigma$ is compact subanalytic. Apply the simultaneous subanalytic triangulation theorem \cite{Shiota1988Triangulation} to the finite family consisting of all $N_p$ and all $\Pi_\sigma$. Each $\Theta_p$ is a finite union of the corresponding flow prisms, and therefore is a subpolyhedron. The same is then true for $\widehat N_p=N_p\cup\Theta_p$ and for every finite union $X_A=\bigcup_{p\in A}\widehat N_p$. An inclusion of subpolyhedra is a cofibration.
	\end{proof}
	
	\subsection{The Prismatic CW Structure}
	
	Here we use flow-subordinate triangulations to construct a different prismatic decomposition, this time of the classifying space $B\cL$. Fix a simplex $\sigma\subset E_p$ and a subset
	$\mathbf r=\set{q_1>_F\cdots>_F q_k}\subset R(\sigma)$. By \Cref{prop:segal}, the product
	$\sigma^\circ\times(\simp^k)^\circ$ lies in the nerve summand indexed by $(p>_F q_1>_F\cdots>_F q_k)$. We denote its image in $B\cL$ by $(\sigma,\set{p}\cup\mathbf r)$; this is homeomorphic to a cell, i.e., an open disk of dimension $\dim\sigma+k$. It will be convenient to write $(\sigma,\set{p})$ for the cell corresponding to each simplex $\sigma\subset N_p$. More generally, we write $(\sigma,\mathbf q)$ for the cell whose {\bf label chain} is $\mathbf q=\set{p}\cup\mathbf r$, so that $\max\mathbf q=p$. 
	
	\begin{lemma}\label{lem:prismatic-faces}
		Given cells $e=(\sigma,\mathbf q)$ and $e'=(\sigma',\mathbf q')$ with
		$q=\max\mathbf q$ and $p=\max\mathbf q'$, one has
		\[
		e\subset\overline{e'} \quad \text{if and only if} \quad \mathbf q\subset\mathbf q' \text{ and }
		\begin{cases}
			\sigma\leq\sigma' & \text{if }q=p,\\
			\sigma\subset\chi_{qp}(\sigma') & \text{if }q<_F p.
		\end{cases}
		\]
	\end{lemma}
	
	\begin{proof}
		Write $\mathbf q'=\set{p=q_0>_F q_1>_F\cdots>_F q_k}$. The closed product
		$\sigma'\times\simp^k$ has its usual product cell structure, whose cells are indexed by a face of $\sigma'$ and a nonempty subset $\mathbf q\subset\mathbf q'$ of the vertices of $\simp^k$. Set $q=\max\mathbf q$. If $q=p$, then the initial vertex survives. All deleted simplex vertices are therefore removed by face maps $d_i$ with $i>0$, which forget contacts by \Cref{prop:segal}. The corresponding cells are consequently $(\sigma,\mathbf q)$ with $\sigma\leq\sigma'$. Suppose instead that $q<_F p$, and write $q=q_j$. Passing to the face with vertex set $\mathbf q$ deletes the initial vertices $q_0,\ldots,q_{j-1}$. Repeated application of $d_0$, together with \Cref{thm:factorisation}, sends the base point $x$ to $\chi_{qp}(x)$. Any further deleted vertices are removed by face maps $d_i$ with $i>0$ and hence do not alter the base point. The resulting face is therefore carried onto the subcomplex $\chi_{qp}(\sigma')\subset\cK_q$, whose cells are precisely the $(\sigma,\mathbf q)$ with $\sigma\subset\chi_{qp}(\sigma')$.
	\end{proof}
	
	Recall that $\Kord(P)$ denotes the order complex of $P$, whose $k$-simplices are strictly descending chains of length $k+1$. 	For a down-set $A\subset P$, let $\cL_A$ denote the full subcategory of $\cL$ whose object space is $\coprod_{p\in A}N_p$. 
	
	\begin{lemma}\label{lem:prismatic-cw}
		The cells $(\sigma,\set{p}\cup\mathbf r)$, indexed by $\sigma\in\cK_p$ and
		$\mathbf r\subset R(\sigma)$, are the open cells of a finite regular CW structure on $B\cL$. Moreover, $
		\Ev\times~\Gamma:B\cL \to X\times \Kord(P)$
		is a topological embedding, and for every down-set $A\subset P$ the subspace $B\cL_A$ is the subcomplex consisting of the cells $(\sigma,\mathbf q)$ with $\mathbf q\subset A$.
	\end{lemma}
	
	\begin{proof}
		Consider a chain $\mathbf q=(q_0>_F\cdots>_F q_k)$ in $P$ with $k\geq1$. By \Cref{prop:segal,lem:triangulations}(i), $N_k\cL(\mathbf q)$ is identified with the subcomplex of $\cK_{q_0}$ consisting of those simplices $\sigma$ for which $\set{q_1,\ldots,q_k}\subset R(\sigma)$; for $k=0$ it is all of $\cK_{q_0}$. Hence $B\cL$ is the  quotient of a finite disjoint union of compact polytopes $\sigma\times\simp^{\mathbf q}$, one for each such chain $\mathbf q$ and each simplex $\sigma$ of the corresponding subcomplex, by the simplicial and semi-simplicial face identifications. In particular, $B\cL$ is compact. Let
		$\kappa_{\sigma,\mathbf q}:\sigma\times\simp^{\mathbf q}\to B\cL$ denote the resulting maps.
		
		{\em 1. Embedding:} We first show that $\Ev\times\Gamma$ is an embedding. Every point of $B\cL$ has a representative $(x,\lambda)$ with all barycentric coordinates positive, obtained by deleting the zero-coordinate vertices through the corresponding face maps. For such a representative, $\Gamma(x,\lambda)$ lies in the relative interior of a unique simplex of $\Kord(P)$, and therefore determines both the chain $q_0>_F\cdots>_F q_k$ and the barycentric coordinates $\lambda$. If $k=0$, evaluation is simply the inclusion $N_{q_0}\hookrightarrow X$, and hence determines $x$. If $k\geq1$, write $z=\Ev(x,\lambda)=\phi_s(x)$. Since all barycentric coordinates are positive, $s>0$. The point $x\in E_{q_0}$ is noncritical, so strict descent of $f$ makes it the unique point on the backward orbit of $z$ lying on the level $\{f=c_{q_0}-\epsilon\}$. Thus $x$ is also determined, proving injectivity. Since $B\cL$ is compact and $X\times \Kord(P)$ is Hausdorff, $\Ev\times\Gamma$ is a topological embedding. 
		
		{\em 2. Regularity:} We also claim that each characteristic map $\kappa_{\sigma,\mathbf q}$ is injective. On the fixed simplex $\simp^{\mathbf q}$, the value of $\Gamma$ determines the barycentric coordinates, including their zero entries. Given the evaluation point, the original $x\in\sigma$ is recovered either directly when the induced flow time is zero, or as the unique point of the backward orbit on the level $\{f=c_{\max\mathbf q}-\epsilon\}$ when that time is positive. Hence $(\Ev\times\Gamma)\circ\kappa_{\sigma,\mathbf q}$, and therefore $\kappa_{\sigma,\mathbf q}$, is injective. Thus $\kappa_{\sigma,\mathbf q}$ embeds the convex polytope $\sigma\times\simp^{\mathbf q}$, a closed ball of dimension $\dim\sigma+|\mathbf q|-1$. Every point has a reduced representative with positive barycentric coordinates, whose base point lies in the relative interior of a unique simplex of the relevant triangulation. It follows that the sets
		$\kappa_{\sigma,\mathbf q}(\sigma^\circ\times(\simp^{\mathbf q})^\circ)$ are pairwise disjoint and cover $B\cL$. 
		Finally, \Cref{lem:prismatic-faces} shows that the boundary of each closed cell is a union of cells of strictly smaller dimension.
		
		{\em 3. Down-Set Filtration:} A cell belongs to $B\cL_A$ exactly when its label chain $\mathbf q$ is contained in $A$. Since every face has label chain contained in $\mathbf q$ by \Cref{lem:prismatic-faces}, these cells form a CW subcomplex, as desired.
	\end{proof}
	
	The regular CW decomposition constructed above depends on the triangulations $\cK_p$ of the isolating blocks $N_p$. We call this the {\bf prismatic} structure induced by the $\cK_p$.
	
	\begin{proposition}\label{prop:order-complex}
		For every down-set $A\subset P$, we have $
		\Gamma^{-1}(\Kord(A))=B\cL_A$.
	\end{proposition}
	
	\begin{proof}
		On the interior of a cell $(\sigma,\mathbf q)$, the map $\Gamma$ lands in the relative interior of the simplex of $\Kord(P)$ with vertex set $\mathbf q$. Hence this cell lies over $\Kord(A)$ precisely when $\mathbf q\subset A$, which by \Cref{lem:prismatic-cw} is precisely the condition for it to belong to $B\cL_A$.
	\end{proof}

	\subsection{Homotopy Type} We are now able to prove Theorem (A) from the Introduction.
	
	\begin{theorem}\label{thm:classifying-space}
		The evaluation map $\Ev:B\cL\to X$
		is a simple homotopy equivalence.
	\end{theorem}
	
	\begin{proof}
		Evaluation is continuous by \Cref{lem:evaluation-well-defined} and surjective by \Cref{prop:transit-lattice}(i). If $z$ belongs to no transit space $\Theta_p$, every representative of $z$ is a zero-simplex. Since the blocks are disjoint, the fibre is then a singleton.
		
		Otherwise choose $z=\phi_t(x)$, with $x\in E_p$ and $0\leq t\leq\tau_p(x)$, whose source label $p$ is maximal among transit representations of $z$. Such a label exists and is unique by finiteness and \Cref{prop:transit-lattice}(ii). Write the complete contact list of this transit as $p=q_0>_F q_1>_F\cdots>_F q_m$, with times $0=t_0<t_1<\cdots<t_m=\tau_p(x)$.
		
		We claim that every representative of $z$ is carried by the simplex on this list. A representative with positive nerve degree determines another transit through $z$. Its source lies on the same orbit as $x$. If its source label is $p$, its source point is $x$, by strict descent on the common exit level. Otherwise maximality of $p$ and \Cref{prop:transit-lattice}(ii) place its source at a later exit contact of the $p$-transit, before or at $z$. Contact factorisation identifies its remaining contacts with the corresponding tail of the list. A zero-simplex representative is either $x$ or a contact point of this transit: indeed the transit meets its own block only at $x$, and every other block it meets is recorded in the list.
		
		Conversely, \Cref{prop:segal} realises the complete list as a nerve simplex. Each nonempty subset is a face. Deleting its initial vertices is legitimate because every contact preceding another contact lies in an exit set, and \Cref{thm:factorisation} identifies the resulting source and tail. Thus all representatives just described are faces of this single simplex. Since $\Ev\times\Gamma$ is an embedding, projection identifies the fibre homeomorphically with
		\[
		\set{(\lambda_0,\ldots,\lambda_m)\in\simp^m\mid
			\sum_{i=0}^m\lambda_it_i=t}.
		\]
		This is a nonempty convex polytope. The source is a finite regular CW complex, hence a finite polyhedron, and the target is a compact triangulated manifold. Every fibre has trivial shape, so evaluation is a {\bf cell-like map} \cite{Lacher1969}. Chapman's cell-like mapping theorem now makes it a simple homotopy equivalence \cite[Theorem~2]{Chapman1973CellLike}; here simple homotopy types are taken with respect to the prismatic structure of \Cref{lem:prismatic-cw} on $B\cL$, which is regular and hence polyhedral, and the triangulation of \Cref{prop:global-triangulation} on $X$.
	\end{proof}
	
	The following result shows that the homotopy equivalence established above is compatible with the down-set filtrations of $B\cL$ and $X$.
	
	\begin{theorem}\label{thm:filtered-classifying}
		For every down-set $A\subset P$, evaluation restricts to a simple homotopy equivalence $
		\Ev_A:B\cL_A\longrightarrow X_A$, and
		these maps are natural under inclusion of down-sets.
	\end{theorem}
	
	\begin{proof}
		The image is $X_A$: every block in $A$ is in the image, and every transit from such a block is evaluated by the edge to its terminal contact, whose label also lies in $A$. If a full evaluation fibre is a singleton, surjectivity ensures that its restriction over $z\in X_A$ is the same singleton. Otherwise consider a fibre of the full evaluation map, identified in the preceding proof with a convex slice of the simplex on a complete contact list
		$q_0>_F\cdots>_F q_m$. By \Cref{prop:order-complex},
		\[
		\Ev_A^{-1}(z)
		=
		\Ev^{-1}(z)\cap\Gamma^{-1}(\Kord(A)).
		\]
		Since $A$ is downward closed, its intersection with the contact chain is either empty or a terminal segment $q_j>_F\cdots>_F q_m$. For $z\in X_A$, surjectivity of the restricted evaluation makes the intersection of the fibre with this face nonempty. Under the simplex coordinates above the restricted fibre is obtained by imposing
		$\lambda_0=\cdots=\lambda_{j-1}=0$. It is therefore again a nonempty convex polytope. By \Cref{lem:prismatic-cw}, $B\cL_A$ is a finite CW subcomplex, while
		\Cref{prop:global-triangulation} makes $X_A$ a finite subpolyhedron of $X$. Another appeal to Chapman's cell-like mapping theorem ensures that
		$\Ev_A:B\cL_A\to X_A$ is a simple homotopy equivalence. Naturality is immediate because every $\Ev_A$ is the restriction of the global $\Ev$.
	\end{proof}

	\section{The Landfall Chain Complex}\label{sec:compression}
	
	In this section we compress the prismatic CW structure on $B\cL$ along the forward flow by means of discrete Morse theory. The surviving cells are the non-exit simplices of the blocks, and their Morse boundary is the landfall differential of Theorem (B). Lifting the same compression to the universal cover recovers Whitehead and Reidemeister torsion. Fix flow-subordinate triangulations $\cK_p$ of a tight block system $\set{\epsilon,N_p}$ as in \Cref{lem:triangulations}.
	
	\subsection{Flow-Induced Acyclic Partial Matching}
	
	Here we obtain a compressed chain model for $B\cL$ by applying discrete Morse theory \cite{Forman1998,dmpers} to its prismatic CW structure. The basic idea is to preserve the simplices of $N_p$ which are not contained in $E_p$ and pair away all other cells along the forward flow. Explicitly, for every $\sigma\subset E_p$ and every $\mathbf r\subset R(\sigma)$ with $\dest(\sigma)\notin\mathbf r$, pair
	\[
	e=(\sigma,\{p\}\cup\mathbf r)
	\quad\text{with}\quad
	e'=(\sigma,\{p,\dest(\sigma)\}\cup\mathbf r).
	\]
	The set $\mathbf r$ may be empty. In particular, each exit simplex is paired with its terminal prism. Let $\Sigma$ denote this collection of cell-pairs.
	
	\begin{lemma}\label{lem:acyclic-matching}
		The collection $\Sigma$ is an acyclic partial matching on the prismatic CW model of $B\cL$. Its critical cells are precisely the block cells $(\sigma,\{p\})$ with $\sigma\not\subset E_p$. Moreover, $\Sigma$ restricts to an acyclic matching on $B\cL_A$ for every down-set $A\subset P$.
	\end{lemma}
	
	\begin{proof}
		By construction, each cell occurs in at most one $\Sigma$-pair. If $\sigma\subset E_p$, then $(\sigma,\{p\}\cup\mathbf r)$ is the lower member of a pair when $\dest(\sigma)\notin\mathbf r$ and the upper member otherwise. All remaining cells have $T=\varnothing$ and $\sigma\not\subset E_p$, so these are precisely the critical cells.
		
		It remains to establish acyclicity. Consider distinct $\Sigma$-pairs $e_i=(\sigma_i,\{p_i\}\cup\mathbf r_i)$ and $e'_i=(\sigma_i,\{p_i,\dest(\sigma_i)\}\cup\mathbf r_i)$ forming a zigzag path of length $k>1$, with each downward step of codimension one:
		\[
		e_1<e'_1>e_2<e'_2>\cdots>e_k<e'_k.
		\]
		Proceeding by contradiction, assume that $e'_k>e_1$. By \Cref{lem:prismatic-faces}, each downward step forces $p_i\geq_F p_{i+1}$, so all of the $p_i$ must equal a single $p\in P$. Thus all $\sigma_i$ lie in the same $\cK_p$. The same face lemma now gives $\sigma_{i+1}\leq\sigma_i$, so all of these simplices must equal a single $\sigma$. After $e_1=(\sigma,\{p\}\cup\mathbf r)$ and $e'_1=(\sigma,\{p,\dest(\sigma)\}\cup\mathbf r)$, the only downward facet with the same source and base simplex which omits $\dest(\sigma)$ is $e_1$ itself. Every other such facet is an upper matched cell, and hence cannot be $e_2$; this contradicts our assumption that the zigzag path above is a cycle. Finally, all labels in a cell lie below its source. If that source belongs to a down-set $A$, both members of its matched pair belong to $B\cL_A$. Hence the matching restricts as asserted.
	\end{proof}
	
	\subsection{The Morse Boundary Operator}
	
	Give every simplex an orientation and work over $\Z$ unless coefficients are specified. For each $p$, let $Q_{p,*}=C_*(N_p,E_p)$ and write $d_p$ for its relative simplicial boundary. We identify $Q_{p,*}$ with the graded subgroup of $C_*(N_p)$ spanned by the simplices not contained in $E_p$; this subgroup need not be a subcomplex. The boundary $\partial\alpha$ of each such simplex $\alpha$ therefore decomposes uniquely as $d_p\alpha+\beta_p\alpha$, where $\beta_p\alpha$ lies in $C_{*-1}(E_p)$. Thus $\beta_p$ records the exit faces discarded by the relative boundary. Set $Q_*:=\bigoplus_p Q_{p,*}$ and $\mathsf E_*:=\bigoplus_p C_*(E_p)$; similarly write $d:=\bigoplus_p d_p$ and $\beta:=\bigoplus_p\beta_p$. 
	If $\sigma\subset E_p$ has destination $q=\dest(\sigma)$, then $\chi_{qp}(\sigma)$ is a subcomplex of $\cK_q$. Orient this image by $\chi_{qp}|_\sigma$ and denote its simplicial fundamental chain by $(\chi_{qp}|_\sigma)_*[\sigma]$.
	
	\begin{lemma}\label{lem:terminal-chain}
		No simplex of $(\chi_{qp}|_\sigma)_*[\sigma]$ lies in $E_q$.
	\end{lemma}
	
	\begin{proof}
		By \Cref{prop:chambers}, the image $\chi_{qp}(\sigma^\circ)$ is disjoint from $E_q$. Hence $\chi_{qp}(\sigma)\cap E_q$ is contained in $\chi_{qp}(\partial\sigma)$, which has dimension less than $\dim\sigma$. Both the image and $E_q$ are subcomplexes, so no simplex of dimension $\dim\sigma$ in the image can therefore lie in $E_q$; these are exactly the simplices occurring in its fundamental chain.
	\end{proof}
	
	\begin{definition}\label{def:landfall-operator}
		The {\bf landfall operator} $L:\mathsf E_*\to Q_*$ sends each oriented simplex $\sigma\subset E_p$ to the chain $(\chi_{qp}|_\sigma)_*[\sigma]$ in $Q_{q,*}$, where $q=\dest(\sigma)$, and extends linearly.
	\end{definition}
	
	Write $L_{qp}:C_*(E_p)	\to Q_{q,*}$ for the $(q,p)$ component of $L$. The operator $L$ has degree zero and strictly lowers the flow order. The preceding lemma allows us to regard its values as chains consisting entirely of critical cells, without discarding any top-dimensional terms.
	
	\begin{lemma}\label{lem:morse-boundary}
		The Morse differential of $\Sigma$ is
		\begin{equation}\label{eq:morse-differential}
			d^{\mathrm{LF}}=d+L\beta.
		\end{equation}
	\end{lemma}
	
	\begin{proof} Let $\alpha\subset N_p$ be an oriented critical simplex. Each facet $\alpha' < \alpha$ in $(N_p-E_p)$ contributes its usual incidence coefficient $[\alpha:\alpha']$, giving $d_p\alpha$. An exit facet $\sigma$ of dimension $k$ is instead paired with $e_\sigma := (\sigma,\{p,q\})$, where $q=\dest(\sigma)$. Orient this prism as $(-1)^k[\sigma]\times[p,q]$, with the edge directed from $p$ to $q$. Its two horizontal faces contribute $-[\sigma]+L\sigma$. Since its incidence with the matched source is $-1$, cancellation contributes $[\alpha:\sigma]L\sigma$ to the Morse boundary of $\alpha$, where $[\alpha:\sigma] \in \set{\pm 1}$ is the induced local orientation. By \Cref{lem:terminal-chain}, every term in this target chain is critical.
		
		It remains to rule out contributions through the side faces of $e_\sigma$. For a cell $(\rho,\mathbf q)$ put $\nu=|\mathbf q|-1$. A reversed matching edge raises $\nu$ by one. A downward facet lowers it by at most one: by \Cref{lem:prismatic-faces}, the label chain can only shrink and the base dimension cannot increase, while the total dimension drops by exactly one. Thus $\nu$ cannot decrease across a complete gradient-path step. A side face of $e_\sigma$ has $\nu=1$. Every continuation through matched cells therefore has $\nu\geq1$, and cannot end at a critical cell, all of which have $\nu=0$. The only critical contributions are consequently $d_p\alpha$ and $L\beta_p\alpha$, as claimed.
	\end{proof}
	
	\begin{definition}\label{def:landfall-complex}
		The {\bf landfall chain complex} is $\mathsf C=(Q_*,d^{\mathrm{LF}})$, with $d^{\mathrm{LF}}=d+L\beta$. For a down-set $A\subset P$, write $\mathsf C_A$ for its subcomplex on $\bigoplus_{p\in A}Q_{p,*}$.
	\end{definition}
	
	\begin{theorem}\label{thm:landfall-complex}
		The landfall chain complex is the cellular Morse complex of $\Sigma$. The Morse contraction is compatible with every down-set subcomplex $B\cL_A$.
	\end{theorem}
	
	\begin{proof}
		The critical generators are exactly those of $Q_*$ by \Cref{lem:acyclic-matching}, and \Cref{lem:morse-boundary} computes their differential. In particular, $(d+L\beta)^2=0$. Every downward cellular step preserves or lowers the source label, while every reversed matching step preserves it. The maps and homotopies in the Morse contraction therefore preserve every down-set.
	\end{proof}
	
	\begin{corollary}\label{cor:filtered-chains}
		For every commutative ring $\Lambda$ there is a zigzag of chain maps, natural over down-sets $A\subset P$, whose components are chain homotopy equivalences
		\[
		\mathsf C_A\otimes_\Z\Lambda
		\simeq C_*^{\mathrm{cell}}(B\cL_A;\Lambda)
		\simeq \Sing_*(B\cL_A;\Lambda)
		\simeq \Sing_*(X_A;\Lambda).
		\]
	\end{corollary}
	
	\begin{proof}
		Use the Morse contraction, the cellular--singular comparison and the evaluation map of \Cref{thm:filtered-classifying}. The first two comparisons can be chosen on the whole finite CW complex and restricted to its down-set subcomplexes. Evaluation is already natural under these inclusions.
	\end{proof}
	
	Thus we have proved Theorem (B) from the Introduction. The diagonal differential is the ordinary relative boundary, and its strictly order-lowering perturbation is $\Delta=L\beta$. The identity $(d+\Delta)^2=0$ includes all the compatibility relations needed for the construction.
	
	\subsection{Torsion}\label{sec:torsion}
	
	Assume that $X$ is connected, choose a basepoint and write $\pi=\pi_1X$. Let $\varpi:\widetilde X\to X$ be the universal cover, and put $\widetilde N_p=\varpi^{-1}(N_p)$ and $\widetilde E_p=\varpi^{-1}(E_p)$. These are full inverse images; they need not be connected. Lift the triangulations $\cK_p$ and choose one lift of each oriented simplex. Then
	\[
	Q_*^\pi=\bigoplus_p C_*(\widetilde N_p,\widetilde E_p)
	\]
	is a finite based free $\Z[\pi]$-module in each degree, generated by the chosen lifts of non-exit simplices. Its relative boundary $d^\pi$ and exit boundary $\beta^\pi$ are defined by the same splitting of the simplicial boundary as before.
	
	The gradient flow lifts uniquely to a $\pi$-equivariant flow $\widetilde\phi$ on $\widetilde X$. For $\widetilde x$ above $x\in\mathcal C_{qp}$, define $\widetilde\chi_{qp}(\widetilde x)=\widetilde\phi_{\tau_{qp}(x)}(\widetilde x)$. This gives an equivariant lift of the contact map. Sending a lifted exit simplex to its terminal contact chain defines the lifted landfall operator $L^\pi$. Expressed in the chosen target lifts, its values take the form
	\[
	L^\pi\widetilde\sigma
	=(\widetilde\chi_{qp}|_{\widetilde\sigma})_*[\widetilde\sigma]
	=\sum_{\omega,g}e_{\omega,g}\,g[\widetilde\omega],
	\qquad q=\dest(\sigma),
	\]
	with finitely many nonzero integers $e_{\omega,g}$. By \Cref{lem:terminal-chain}, the target simplices are again non-exit simplices. Recall that a chain homotopy equivalence of based free $\Z[\pi]$-complexes is simple if its Whitehead torsion vanishes in  $\operatorname{Wh}(\pi)$ \cite{Milnor1966,Cohen1973}.
	
	\begin{theorem}\label{thm:equivariant-landfall}
		The complex $\mathsf C^\pi=(Q_*^\pi,d^\pi+L^\pi\beta^\pi)$ is simple chain homotopy equivalent to $C_*^{\mathrm{cell}}(\widetilde X)$. More precisely, there is a zigzag of filtration-preserving chain maps which restricts, for every down-set $A\subset P$, to simple chain homotopy equivalences
		\[
		\mathsf C_A^\pi
		\simeq_s C_*^{\mathrm{cell}}(\widetilde{B\cL_A})
		\simeq_s C_*^{\mathrm{cell}}(\widetilde X_A).
		\]
		Here $\pi_1(B\cL)$ is identified with $\pi$ by evaluation, and tildes denote full inverse images in the corresponding universal covers.
	\end{theorem}
	
	\begin{proof}
		Lift the prismatic CW structure and the matching $\Sigma$ equivariantly to $\widetilde{B\cL}$. There are finitely many orbits of cells and matched pairs. Each matched incidence is $\pm1$ downstairs and hence $\pm g$ upstairs for some $g\in\pi$. These are units in $\Z[\pi]$, so algebraic discrete Morse cancellation \cite{Skoldberg2006,JollenbeckWelker2009} applies. The critical orbits are the lifts of non-exit block simplices, and the calculation in \Cref{lem:morse-boundary} gives $d^\pi+L^\pi\beta^\pi$.
		Each cancellation has torsion represented by a unit, which vanishes in $\operatorname{Wh}(\pi)$. The Morse equivalence is therefore simple. Evaluation is a simple homotopy equivalence by \Cref{thm:classifying-space}, so its cellular lift gives the second simple chain equivalence. Since the Morse contraction preserves every down-set, a cellular approximation to evaluation may also be chosen to preserve these finitely many subcomplexes. Restricting to a down-set and applying \Cref{thm:filtered-classifying} proves the filtered assertion. On each component, extension of coefficients to $\Z[\pi]$ preserves the vanishing of Whitehead torsion.
	\end{proof}
	
	Let $\kk$ be a field. A bounded acyclic complex of finite-dimensional $\kk$-vector spaces
	with chosen bases has a well-defined (Reidemeister) {\bf torsion} $\RT(C_*)\in\kk^\times$
	\cite[\S 3]{Milnor1966}; we normalise so that \[0\to\kk\xrightarrow{a}\kk\to0,\] with the
	source of $a$ in degree one, has torsion $a^{-1}$. Regard the row-vector space $\kk^m$ as a right $\Z[\pi]$-module $\kk_\rho^m$ by $v\cdot g=v\rho(g)$.
	
	\begin{corollary}\label{cor:landfall-torsion}
		If $H_*(X;\kk_\rho^m)=0$, then
		\[
		\RT_\rho(X)=\RT\bigl(\kk_\rho^m\otimes_{\Z[\pi]}\mathsf C^\pi\bigr)
		\quad\bmod\{\pm\det\rho(g)\mid g\in\pi\}.
		\]
		The same conclusion holds for every acyclic down-set or interval quotient of the filtered complex.
	\end{corollary}
	
	\begin{proof}
		Tensor the simple chain equivalences of \Cref{thm:equivariant-landfall} with $\kk_\rho^m$. For interval quotients, use additivity of Whitehead torsion for the corresponding based short exact sequences. Changing the chosen lift or orientation of a cell gives exactly the stated indeterminacy.
	\end{proof}
	
	Torsion has previously been extracted from Morse--Bott--Smale geometric complexes and from filtrations associated to Bott integrals \cite{BurgheleaHaller2004,FelSanchezMorgado1999}. Here the based free complex comes directly from bounded-time contacts of the original analytic gradient flow, with no local normal form or transversality assumption.
	
	\section{Connection Matrices}\label{sec:linearisation}
	
	We now extract the standard invariants of Conley theory from the down-set filtration of $X$ and its landfall models: the interval Conley indices, their independence from the chosen blocks, the Morse--Conley spectral sequence and, after homological perturbation, connection matrices. The section ends with the Morse--Bott case, where no transversality is required.
	
	\subsection{Interval Conley Indices}\label{subsec:index-filtration}
	
	Recall the compact positively invariant sets $X_A$ from \Cref{prop:transit-lattice}. A subset $I\subset P$ is an {\bf interval} if $q<_F r<_F p$ with $q,p\in I$ forces $r\in I$. Every interval has the form $I=A-A'$ where
	\[
	A = ~ \downarrow I := \set{p \in P \mid p \leq_F q \text{ for some } q \in I},
	\] and set $A'=A-I$. Write $M(I)$ for the union of the complete orbits whose two limiting critical components belong to $I$.
	
	\begin{proposition}\label{thm:index-filtration}
		For every pair of down-sets $A'\subset A$, the pair $(X_A,X_{A'})$ is an index pair for $M(A-A')$.
	\end{proposition}
	
	\begin{proof}
		Set $I=A-A'$. No transit enters a block interior, so $X_A\cap\Crit(f)=\coprod_{p\in A}S_p$, and the block interiors with labels outside $A$ are disjoint from $X_A$. Those with labels in $A'$ lie inside $X_{A'}$. Thus any complete orbit in $\overline{X_A-X_{A'}}$ has both limiting labels in $I$. Conversely, an orbit with both limits in $I$ belongs to $X_A$ by its backward limit and the saturation formula. It cannot enter $X_{A'}$, since positive invariance and closedness would place its forward limit there. Hence $\Inv(\overline{X_A-X_{A'}})=M(I)$. For $x\in M(I)$, choose $p\in I$ and $T>0$ with $\phi_{-T}(x)\in N_p^\circ$. The open set $N_p^\circ-\phi_T^{-1}(X_{A'})$ contains $\phi_{-T}(x)$, and its image under $\phi_T$ is a neighbourhood of $x$ contained in $X_A-X_{A'}$. This proves the isolation condition. Both spaces are compact and positively invariant, so the relative invariance and exit conditions are automatic.
	\end{proof}
	
	Thus $A\mapsto X_A$ is an index filtration in the sense of \cite{Franzosa1986,RobbinSalamon1992}. The filtered evaluation equivalence identifies all its quotient indices at once.
	
	\begin{theorem}\label{thm:homotopy-conley}
		For nested down-sets $A'\subset A$, evaluation induces a pointed homotopy equivalence
		\[
		B\cL_A/B\cL_{A'}\longrightarrow X_A/X_{A'}\simeq h(M(A-A')).
		\]
		These equivalences respect the cofibre sequences associated to triples of nested down-sets.
	\end{theorem}
	
	\begin{proof}
		The source and target pairs are cofibration pairs by \Cref{lem:prismatic-cw,prop:global-triangulation}. Evaluation is a homotopy equivalence on each member by \Cref{thm:filtered-classifying}, hence also on the quotients. Apply \Cref{thm:index-filtration}. Compatibility with cofibre sequences follows from naturality of evaluation.
	\end{proof}
	
	For $I=A-A'$, set $\mathsf C(I)=\mathsf C_A/\mathsf C_{A'}$; as a graded module this is $\bigoplus_{p\in I}Q_{p,*}$ with the induced differential, so it does not depend on the choice of nested down-sets with difference $I$. The filtered chain comparison gives $H_*(\mathsf C(I);\Lambda)\cong\Con_*(M(I);\Lambda)$, compatibly with the maps and connecting homomorphisms for nested down-sets. In Franzosa's terminology, this identifies the homology braid of $\mathsf C$ with the homology index braid \cite{Franzosa1989}. The same quotient equivalences give the corresponding statements for any reduced generalised homology theory.
	
	\subsection{Independence of the Blocks}\label{subsec:block-independence}
	
	We now compare two tight block systems for the same flow. Denote them by $N$ and $N'$, and write $X_A(N)$ and $X_A(N')$ for their forward saturations. A map or homotopy between these filtered spaces is called {\bf filtered} if it carries the stage indexed by $A$ into the stage indexed by $A$ for every down-set $A$.
	
	\begin{proposition}\label{prop:block-independence}
		For all sufficiently large $T$, the flow map $\phi_T:X\to X$ defines a filtered homotopy equivalence from $X_\bullet(N)$ to $X_\bullet(N')$. Its filtered homotopy class $\Phi_{N,N'}$ is independent of $T$, and these classes satisfy
		\[
		\Phi_{N,N}=1,\qquad
		\Phi_{N',N''}\Phi_{N,N'}=\Phi_{N,N''}.
		\]
	\end{proposition}
	
	\begin{proof}
		Fix a down-set $A$. Every $x\in X_A(N)$ has forward limit in some $S_p$ with $p\in A$, since $X_A(N)$ is closed and positively invariant and its critical points have labels in $A$. Hence $\phi_{t_x}(x)\in(N'_p)^\circ$ for some $t_x\geq0$. By continuity, a neighbourhood $U_x$ of $x$ also enters $(N'_p)^\circ$ at time $t_x$. Positive invariance of $X_A(N')$ then gives $\phi_t(U_x\cap X_A(N))\subset X_A(N')$ for every $t\geq t_x$. A finite subcover of the compact set $X_A(N)$ gives a uniform time $T_A$. There are finitely many down-sets, so one time $T$ works for all of them. Interchanging the systems gives a filtered map $\phi_{T'}$ in the reverse direction. Both composites are $\phi_{T+T'}$, and the homotopies $\phi_{s(T+T')}$, $0\leq s\leq1$, remain in the appropriate source stages by positive invariance. Thus the two maps are filtered homotopy inverses. Interpolating between any two sufficiently large times gives a filtered homotopy between the comparison maps. The identity assertion follows from positive invariance, and the composition assertion follows from addition of flow times.
	\end{proof}
	
	Together with evaluation, these comparisons give coherent isomorphisms between the down-set diagrams $B\cL_\bullet(N)$ in the category of down-set diagrams localised at objectwise homotopy equivalences. They induce the identity on $H_*(X)$, since every flow map is homotopic to the identity. This applies the flow-defined comparison of Conley theory \cite{Salamon1985} simultaneously to all down-sets, without comparing individual object or morphism spaces.
	
	\subsection{The Morse--Conley Spectral Sequence}\label{subsec:spectral-sequence}
	
	Fix a function $\ell:P\to\{0,\ldots,m\}$ such that $q<_F p$ implies $\ell(q)<\ell(p)$. The sets $A_i=\{p\mid\ell(p)\leq i\}$ are down-sets, giving a finite filtration $\varnothing=X_{A_{-1}}\subset X_{A_0}\subset\cdots\subset X_{A_m}=X$.
	
	\begin{theorem}\label{thm:spectral-sequence}
		For every commutative ring $\Lambda$, this filtration has a homological spectral sequence
		\begin{equation}\label{eq:spectral-sequence}
			E^1_{i,j}\cong\bigoplus_{\ell(p)=i}\Con_{i+j}(S_p;\Lambda)
			\quad\Longrightarrow\quad H_{i+j}(X;\Lambda),
		\end{equation}
		with $d^r:E^r_{i,j}\to E^r_{i-r,j+r-1}$. For fixed $\ell$, its pages and differentials from $E^1$ onward are canonically independent of the tight block system and flow-subordinate triangulations. The spectral sequences of the filtered nerve and the landfall complex are naturally isomorphic to it from $E^1$ onward. In the landfall model, $d^1$ is induced by the part of $L\beta$ which lowers $\ell$ by one.
	\end{theorem}
	
	\begin{proof}
		Filter $\Sing_*(X;\Lambda)$ by $\Sing_*(X_{A_i};\Lambda)$ and the landfall complex by $\mathsf C_{A_i}\otimes\Lambda$. The comparisons of \Cref{cor:filtered-chains} induce isomorphisms on the homology of successive quotients, and therefore identify the spectral sequences from $E^1$ onward. On the landfall side, each successive quotient is the direct sum of $(C_*(N_p,E_p;\Lambda),d_p)$ over $\ell(p)=i$: every term of $L\beta$ strictly lowers $\ell$. This gives \eqref{eq:spectral-sequence} and the description of $d^1$. The filtration is finite, so $E^\infty_{i,j}$ is the $i$th associated graded group of $H_{i+j}(X;\Lambda)$ with its induced filtration. 	
		The flow comparisons in \Cref{prop:block-independence} are filtered homotopy equivalences and hence induce spectral-sequence isomorphisms from $E^1$ onward. Their time independence and composition law give the asserted canonical identifications. The comparison with singular chains removes the triangulation choice as well.
	\end{proof}
	
	This is the classical Morse--Conley spectral sequence; existence and index-filtration independence are treated in \cite[Theorems~7.1--7.2]{Bartlomiejczyk2005}. The landfall construction supplies a finite chain model with differential $d+L\beta$. Neither field coefficients nor torsion-free local homology are needed.
	
	\begin{remark}\label{rem:level-choice}
		A linear extension $p_0,\ldots,p_m$ of $P$ gives $\ell(p_i)=i$, so each summand of the first page is a single local Conley group. Changing the extension can change the page on which a differential appears. A choice depending only on $P$ is the height function
		\[
		\ell(p)=\max\{k\mid p_0<_F\cdots<_F p_k=p\}.
		\]
		Alternatively, the position of $c_p$ among the distinct critical values gives a canonical choice for the given $f$, grouping components at the same value. Block independence holds for each fixed choice of $\ell$.
	\end{remark}
	
	\subsection{Connection Matrices}\label{subsec:connection-matrices}
	
	Work over a commutative ring $\Lambda$. Write $Q_p^\Lambda=C_*(N_p,E_p;\Lambda)$ and $Q_*^\Lambda=\bigoplus_p Q_p^\Lambda$. By \Cref{lem:block-properties}, the homology of $Q_p^\Lambda$ is $\Con_*(S_p;\Lambda)$. The landfall complex $\mathsf C^\Lambda=\mathsf C\otimes\Lambda$ has differential $d+\Delta$, where $d=\bigoplus_p d_p$ and $\Delta=L\beta$ strictly lowers the flow order. We now transfer this perturbation to local Conley homology.
	
	\begin{definition}\label{def:contraction}
		A {\bf normalised contraction} of $(C,d)$ onto $(V,d_V)$ consists of chain maps $\jmath:V\to C$, $\pi:C\to V$ and a degree $+1$ map $h$ satisfying
		\[
		\pi\jmath=1,\qquad dh+hd=\jmath\pi-1,
		\qquad h^2=h\jmath=\pi h=0.
		\]
	\end{definition}
	
	\begin{proposition}\label{prop:contractions}
		Every bounded complex of finite-dimensional vector spaces over a field admits a normalised contraction onto its homology with zero differential. A bounded complex of finitely generated free abelian groups admits such a contraction over $\Z$ if and only if all its homology groups are free abelian.
	\end{proposition}
	
	\begin{proof}
		Over a field, choose splittings $C_n=B_n\oplus H_n\oplus T_n$ with $d:T_n\to B_{n-1}$ an isomorphism. Let $\jmath$ include $H_n$, let $\pi$ project to it, and let $-h$ invert $d$ on the boundary summand and vanish elsewhere. Over $\Z$ the same construction works when the homology groups are free: the sequences $0\to B_n\to Z_n\to H_n\to0$ then split, as do $0\to Z_n\to C_n\to B_{n-1}\to0$, since every $B_{n-1}$ is free. Conversely, a contraction makes each homology group a direct summand of a free abelian chain group, so it must be free.
	\end{proof}
	
	\begin{definition}\label{def:connection-matrix}
		A {\bf connection matrix} is a strictly flow-order-lowering differential $D$ of degree $-1$ on $V=\bigoplus_{p\in P}\Con_*(S_p;\Lambda)$ whose homology braid \cite[\S 3]{Franzosa1989} is naturally isomorphic to the homology index braid of the flow-ordered Morse decomposition.
	\end{definition}
	
	\noindent This is the $P$-graded form of Franzosa's definition \cite{Franzosa1989}; compare \cite{HarkerMischaikowSpendlove2021}.
	
	\begin{theorem}\label{thm:connection-matrix}
		Suppose that a normalised contraction $(\jmath_p,\pi_p,h_p)$ of $Q_p^\Lambda$ onto $V_p=H_*(Q_p^\Lambda)$ with zero differential has been chosen for every $p$. Write $\jmath,\pi,h$ for their direct sums. Then
		\begin{equation}\label{eq:connection}
			D=\pi\Delta(1-h\Delta)^{-1}\jmath
		\end{equation}
		is a connection matrix. Its $(q,p)$ entry is the sum over all chains $(p=p_0>_F p_1>_F\cdots>_F p_r=q)$ of the terms
		\begin{equation}\label{eq:itinerary}
			\pi_q~\left(\Delta_{p_rp_{r-1}}~h_{p_{r-1}}~
			\Delta_{p_{r-1}p_{r-2}}~\cdots
			h_{p_1}~\Delta_{p_1p_0}\right)~\jmath_p,
		\end{equation}
		where $\Delta_{qp}=L_{qp}\beta_p$ and the term for $r=1$ is $\pi_q\Delta_{qp}\jmath_p$.
	\end{theorem}
	
	\begin{proof}
		Every factor of $\Delta$ strictly lowers the finite flow order, whereas $h$ preserves block labels. Hence $h\Delta$ is nilpotent. The finite homological perturbation lemma \cite{Gugenheim1972,GugenheimLambeStasheff1991,BarnesLambe1991} contracts $(Q_*^\Lambda,d+\Delta)$ onto $(V,D)$ with $D$ as in \eqref{eq:connection}. All perturbed maps and homotopies preserve down-sets and descend to interval quotients. By \Cref{thm:homotopy-conley}, the resulting homology braid is therefore the homology index braid, so $D$ is a connection matrix. Expanding the finite geometric series gives \eqref{eq:itinerary}. The same filtered contraction identifies the spectral sequence of $(V,D)$ with \eqref{eq:spectral-sequence} from $E^1$ onward.
	\end{proof}
	
	We call $(V,D)$ the {\bf connection complex} of the chosen block system, triangulations and local contractions.
	
	\begin{corollary}\label{cor:coefficients}
		Over every field $\kk$, the landfall construction produces a connection matrix. Over $\Z$, the finite free landfall complex exists without additional hypotheses, while the local contractions onto homology exist if and only if every local homological Conley group is free abelian.
	\end{corollary}
	
	\begin{proof}
		Apply \Cref{prop:contractions}.
	\end{proof}
	
	\begin{corollary}\label{cor:adjacent}
		Under the hypotheses of \Cref{thm:connection-matrix}, if $q<_F p$ is a covering relation, then $D_{qp}=\pi_qL_{qp}\beta_p\jmath_p$. This gives the attractor--repeller connecting homomorphism
		\[
		\Con_n(S_p;\Lambda)\longrightarrow\Con_{n-1}(S_q;\Lambda),
		\]
		with the usual homological convention for a short exact sequence.
	\end{corollary}
	
	\begin{proof}
		There is no intermediate flow index, so only the first-order term of \eqref{eq:connection} survives. The short exact sequence
		\[
		0\longrightarrow Q_q^\Lambda
		\longrightarrow\mathsf C^\Lambda(\{q,p\})
		\longrightarrow Q_p^\Lambda\longrightarrow0
		\]
		has off-diagonal boundary $\Delta_{qp}=L_{qp}\beta_p$. Its connecting map is induced by this boundary, and the local contractions identify it with $\pi_q\Delta_{qp}\jmath_p$. By \Cref{thm:homotopy-conley}, this is the connecting map of the attractor--repeller index triple.
	\end{proof}
	
	\begin{remark}
		Over $\R$ or $\C$, give the relative simplex basis its Euclidean or Hermitian inner product. Eckmann's finite-dimensional Hodge decomposition \cite{Eckmann1944} then supplies the local contractions: take $\jmath_p$ to include harmonic representatives, $\pi_p$ to be harmonic projection and $h_p=-d_p^*G_p$, where $G_p$ is the inverse of the Laplacian on the orthogonal complement of the harmonic chains. The minus sign agrees with \Cref{def:contraction}.
	\end{remark}
	
	\begin{proposition}\label{prop:gauge}
		Suppose two families of local contractions induce the identity on local homology and produce differentials $D,D'$ on the same graded module $V$. Then there is a chain isomorphism $U:(V,D)\to(V,D')$ of the form $U=1+\Xi$, with $\Xi$ strictly order-lowering. In particular, $D'=UDU^{-1}$.
	\end{proposition}
	
	\begin{proof}
		Let $\jmath_\Delta$ be the perturbed inclusion for the first contraction and $\pi'_\Delta$ the perturbed projection for the second. Their composite $U=\pi'_\Delta\jmath_\Delta$ is a chain map. Its diagonal part is the identity, and every correction term contains a factor of $\Delta$, so $U=1+\Xi$ with $\Xi$ strictly order-lowering. Since $P$ is finite, $\Xi$ is nilpotent and $U^{-1}=1-\Xi+\Xi^2-\cdots$ is a finite sum.
	\end{proof}
	
	\begin{corollary}\label{cor:field-uniqueness}
		Over a field, any two landfall connection matrices for the Morse decomposition \eqref{eq:morse-decomposition} are filtered chain isomorphic, even when they are built from different tight block systems, flow-subordinate triangulations and local contractions.
	\end{corollary}
	
	\begin{proof}
		Each matrix gives a strict $P$-graded differential vector space. By \Cref{thm:homotopy-conley}, its Cartan--Eilenberg system in the sense of \cite{SpendloveVandervorst2025} is the homology index system. The flow comparisons of \Cref{prop:block-independence} identify these systems for different block choices. The result now follows from the classification of strict $P$-graded differential vector spaces by their Cartan--Eilenberg systems \cite[Thm.~5.1 and \S 7]{SpendloveVandervorst2025}.
	\end{proof}
	
	Franzosa's existence theorem starts with a chain complex braid \cite{Franzosa1989}, while modern computational constructions reduce a supplied filtered chain complex algebraically \cite{HarkerMischaikowSpendlove2021,TorrasCasasYimPennig2025}. Here the order-lowering perturbation $\Delta=L\beta$ comes from bounded-time contact geometry. Only the final passage from relative block chains to local homology uses homological perturbation.
	
	\subsection{Morse--Bott Functions}
	
	Suppose that the analytic function $f$ is also Morse--Bott, with connected critical submanifolds $S_p$. Write $\lambda_p$ for the rank of the negative normal bundle $\nu_p^-$ and $o_p$ for its integral orientation local system.
	
	\begin{corollary}\label{cor:morse-bott}
		Over a field $\kk$, the landfall connection complex has underlying graded vector space $V_n=\bigoplus_{p\in P}H_{n-\lambda_p}(S_p;o_p\otimes\kk)$, differential given by \eqref{eq:itinerary}, and homology $H_*(X;\kk)$. No Morse--Bott--Smale transversality is required. The same statement holds over $\Z$ if every $H_*(S_p;o_p)$ is free abelian.
	\end{corollary}
	
	\begin{proof}
		The local index pair is homotopy equivalent to the disk and sphere bundle pair of $\nu_p^-$, by the Morse--Bott lemma. The Thom isomorphism gives the stated local groups. Apply \Cref{thm:connection-matrix} over a field, and \Cref{cor:coefficients} over $\Z$. Compare the transverse trajectory-space construction in \cite{Latschev2000}.
	\end{proof}
	
	For Morse functions this gives an integral complex with one generator per critical point in its Morse index, without requiring the analytic metric to be Morse--Smale.
	
	\section{Examples and Special Cases}\label{sec:examples}
	
	We now provide three examples of landfall data. The first is a degenerate saddle on the sphere whose exit set contains a genuine landfall wall; the second is a family of lens spaces whose torsion is retained by the landfall complex but forgotten by its connection complex; and the third is a pair of ruled surfaces whose top attaching maps are visible in the filtered landfall nerve but not in the homology braid.
	
	\subsection{A Degenerate Sphere Model}
	
	\begin{example}\label{ex:monkey}
		Let $f(x,y,z)=x^3-3xy^2$ on the round sphere $S^2$. Writing $x=\cos\psi\cos\theta$, $y=\cos\psi\sin\theta$ and $z=\sin\psi$ gives $f=\cos^3\psi\cos3\theta$. There are three equatorial maxima at value $1$, three equatorial minima at value $-1$, and two poles $P_\pm$ at value $0$. The extrema are nondegenerate, whereas at either pole $(x,y)$ are analytic coordinates in which $f=\operatorname{Re}(x+iy)^3$. A small isolating disk at a pole has three disjoint exit arcs. Thus $\Con_1(P_\pm;\kk)=\kk^2$, with all other local groups zero, and the total local Poincar\'e polynomial is $3+4t+3t^2$. The negative gradient equations are
		\begin{equation}\label{eq:monkey-flow}
			\dot\psi=3\cos^2\psi\sin\psi\cos3\theta,
			\qquad \dot\theta=3\cos\psi\sin3\theta.
		\end{equation}
		Away from the invariant meridians and the equator, $\frac{d}{d\theta}\log|\tan\psi|=\cot3\theta$, giving the first integral
		\begin{equation}\label{eq:monkey-integral}
			|\tan\psi|=a|\sin3\theta|^{1/3},\qquad a>0.
		\end{equation}
		Equatorial trajectories correspond to $a=0$. The meridians $\theta=k\pi/3$ alternate between maximum--pole and pole--minimum separatrices.
		
		Consider trajectories in the northern sector $0<\theta<\pi/3$. Their minimum distance from $P_+$ is $\arctan(a^{-1})$. Choose $0<d_-<d_+$ so that the closed disk of radius $d_-$ lies in $N_{P_+}^{\circ}$ and $N_{P_+}$ lies in the open disk of radius $d_+$. Trajectories with $a>\cot d_-$ enter the pole block, whereas those with $a<\cot d_+$ miss it and land at the minimum. The maximum--pole separatrix is the limit $a\to\infty$.
		
		On the corresponding arc of the maximum exit set, the pole-landing family is therefore nonempty and open, with a nonempty complement. At a boundary point between these families, compactness of the contact domain retains contact with the pole block, while openness of entrance into its interior excludes terminal landfall there. The trajectory continues to the minimum. This is a genuine landfall wall in the sense of \Cref{cor:wall}: its nonterminal contact factors the maximum-to-minimum arrow through the pole block, producing a $2$-simplex in the landfall nerve. The wall continuation is recorded by its faces and enters the cellular boundary through $L\beta$.
		
		The connection complex has dimensions $3,4,3$ in successive degrees. Its homology is that of $S^2$, so both differentials have rank $2$. Direct contacts from maxima to minima remain present in the category, although the corresponding block of the connection matrix vanishes for degree reasons.
	\end{example}
	
	\subsection{Lens Spaces and Torsion}\label{subsec:lens-torsion}
	
	\begin{example}\label{ex:lens-torsion}
		Fix coprime integers $i>1$ and $j$, let $\zeta=e^{2\pi\sqrt{-1}/i}$ be the corresponding primitive $i$-th root of unity, and equip $L(i,j)=S^3/\langle g\rangle$ with the quotient round metric, where $g(z_1,z_2)=(\zeta z_1,\zeta^jz_2)$. The function $u=|z_2|^2$ has critical circles $C_-=\{z_2=0\}/\langle g\rangle$ and $C_+=\{z_1=0\}/\langle g\rangle$. Its negative gradient satisfies $\dot u=-4u(1-u)$ and preserves both angular coordinates.
		
		Set $f=3u^2-2u^3$. Since $df=6u(1-u)\,du$, this has the same unparameterised trajectories and contact maps as $u$, but its critical circles are normally quartically degenerate. The degeneracy here is only a change of speed along a Morse--Bott flow.
		
		For small $a>0$, take the solid-torus blocks $N_-=\{u\le a\}$ and $N_+=\{u\ge1-a\}$, with block parameter $\epsilon=3a^2-2a^3$. Every point of $E_+=\{u=1-a\}$ lands in $N_-$. In angular coordinates measured in turns, the contact map is the identity on
		\[
		\R^2/\Lambda,\qquad
		\Lambda=\Z^2+\Z\left(\frac1i,\frac ji\right).
		\]
		Thus the contact torus retains the quotient lattice, although the flow order has only two elements.
		
		Write $\pi_1L(i,j)=\langle t\mid t^i=1\rangle$, with $t$ corresponding to $g$, and choose $k$ with $jk\equiv1\pmod i$. The positively oriented cores $C_-$ and $C_+$ represent $t$ and $t^k$: a fundamental angular displacement $1/i$ lifts to $g$ on the first circle and $g^k$ on the second.
		
		For the character $\rho_n(t)=\zeta^n$, $1\le n<i$, the minimum block collapses simply onto its core, whose reduced cellular model has boundary $\zeta^n-1$ from degree one to degree zero. The maximum pair is $(S^1\times\Disk^2,S^1\times S^1)$; its relative product model is the degree-two shift of the circle complex with boundary $\zeta^{nk}-1$. Product subdivisions give simple comparisons with the chosen triangulations. Both twisted complexes are acyclic, and multiplicativity along the two-stage filtration gives
		\begin{equation}\label{eq:lens-torsion}
			\RT_{\rho_n}(L(i,j))=
			\frac{1}{(\zeta^n-1)(\zeta^{nk}-1)}
			\quad\bmod\{\pm\zeta^s\mid s\in\Z\}.
		\end{equation}
		
		By contrast, the local integral Conley groups are $H_*(S^1;\Z)$ and $H_{*-2}(S^1;\Z)$. The connection complex has one generator in each degree from zero to three, and its only possible nonzero differential is $D_2=\pm i$, as follows from $H_1(L(i,j);\Z)=\Z/i$ and $H_2(L(i,j);\Z)=0$. It is independent of $j$.
		
		For instance, $L(7,1)$ and $L(7,2)$ are homotopy equivalent but not simple homotopy equivalent \cite[\S\S 29--30]{Cohen1973}. To see the latter distinction in \eqref{eq:lens-torsion}, put $a_s=2\sin(s\pi/7)$ for $s=1,2,3$. The distinct absolute torsions are $\{a_1^{-2},a_2^{-2},a_3^{-2}\}$ for $L(7,1)$ and $\{(a_1a_2)^{-1},(a_1a_3)^{-1},(a_2a_3)^{-1}\}$ for $L(7,2)$. Their largest elements differ. Changing the fundamental-group generator merely permutes the characters, so this distinction survives all choices. The calculation recovers classical torsion from the landfall filtration and exhibits information lost by the ordinary connection complex.
	\end{example}
	
	\subsection{Attaching Data Beyond the Homology Braid}\label{sec:attaching-example}
	
	\begin{remark}\label{rem:ruled-surfaces}
		The ruled surfaces $\mathbb F_0=S^2\times S^2$ and $\mathbb F_1=\C P^2\#\overline{\C P^2}$ give a classical distinction at the level of attaching maps. Choose generic moment-map components with nonadjacent minimum and maximum vertices in their quadrilateral moment polygons, and torus-invariant analytic K\"ahler metrics. Each function is Morse with critical indices $0,2,2,4$. Its gradient flow comes from the complexified torus and preserves torus orbits. The polygon edges give the cover relations $p_3>_F p_1>_F p_0$ and $p_3>_F p_2>_F p_0$; the open orbit runs from maximum to minimum. Thus both flow orders are the same diamond.
		
		The local integral Conley groups are free and concentrated in even degrees, so the integral connection matrices vanish. By \Cref{thm:homotopy-conley,cor:coefficients}, the full homology index braids are consequently isomorphic. Nevertheless, the final attachments differ. Indeed, put $A'=P-\{p_3\}$. The minimum stage is contractible, and adjoining the two index-two pieces gives successive cofibres $S^2$. Van Kampen and the homology calculation show that the lower stage is simply connected with reduced homology $\Z^2$ in degree two. Hurewicz and Whitehead therefore identify it with $S^2\vee S^2$. The final cofibre is $S^4$; relative Hurewicz and Whitehead similarly give
		\[
		B\cL_i\simeq (S^2\vee S^2)\cup_{\alpha_i}e^4,
		\qquad
		\alpha_0=[u_1,u_2],\quad
		\alpha_1=u_1\circ\eta-u_2\circ\eta,
		\]
		where $u_1,u_2$ are the wedge inclusions and $\eta:S^3\to S^2$ is the Hopf map. The displayed classes follow from Hilton's decomposition of $\pi_3(S^2\vee S^2)$ from \cite{Hilton1955} and the intersection forms: the hyperbolic form for $\mathbb F_0$ and $\operatorname{diag}(1,-1)$ for $\mathbb F_1$. Their different parity precludes an equivalence of attaching classes. The identification of the lower stage with $S^2\vee S^2$ uses an arbitrary basis of its second homology, so the $\alpha_i$ are only defined up to self-equivalences of $S^2\vee S^2$ which need not respect the two separately labelled index-two stages; parity is invariant under such changes of basis, so the conclusion is unaffected.
		
	\end{remark}
	
	\raggedbottom
	\begingroup
	\linespread{1}\selectfont
	\bibliographystyle{abbrv}
	\bibliography{refs}
	\endgroup
	
\end{document}